\documentclass{article}
\usepackage[utf8]{inputenc}
\usepackage[T1]{fontenc}
\usepackage{textcomp}
\usepackage[english]{babel}
\usepackage{amsmath}
\usepackage{amsthm}
\usepackage{amssymb}
\usepackage[all]{xy}

\makeatletter
\theoremstyle{plain}
\newtheorem{thm}{Theorem}[section] 
\theoremstyle{definition}
\newtheorem{defn}[thm]{Definition} 
\newtheorem{lem}[thm]{Lemma}
\newtheorem{note}[thm]{Note}
\newtheorem{prop}[thm]{Proposition}
\newtheorem{rem}[thm]{Remark}
\newtheorem{cor}[thm]{Corollary}
\newtheorem*{thm*}{Theorem}
\newtheorem*{notation*}{Notation}

\@ifundefined{date}{}{\date{}}
\usepackage{bbm}
\usepackage{geometry}
\makeatother

\begin{document}
\global\long\def\spec{\mathrm{Spec}}
\global\long\def\spa{\mathrm{Spa}}
\global\long\def\spv{\mathrm{Spv}}
\global\long\def\spf{\mathrm{Spf}}
\global\long\def\supp{\mathrm{Supp}}
\global\long\def\id{\mathrm{id}}
\global\long\def\Id{\mathrm{Id}}
\global\long\def\Ha{\mathrm{Ha}}
\global\long\def\Hdg{\mathrm{Hdg}}
\global\long\def\ig{\mathcal{IG}}
\global\long\def\IG{\mathfrak{IG}}
\global\long\def\dlog{\mathrm{dlog}}
\global\long\def\Hom{\mathrm{Hom}}
\global\long\def\frob{\mathrm{Frob}}
\global\long\def\Res{\mathrm{Res}}
\global\long\def\RM{\mathrm{RM}}
\global\long\def\dR{\mathrm{dR}}
\global\long\def\Tate{{\rm Tate}}

\global\long\def\c{\mathfrak{c}}
\global\long\def\g{\mathfrak{g}}
\global\long\def\p{\mathfrak{p}}
\global\long\def\q{\mathfrak{q}}
\global\long\def\m{\mathfrak{m}}
\global\long\def\Q{\mathbb{Q}}
\global\long\def\Qp{\Q_{p}}
\global\long\def\Z{\mathbb{Z}}
\global\long\def\Zp{\Z_{p}}
\global\long\def\F{\mathbb{F}}
\global\long\def\Fp{\F_{p}}
\global\long\def\H{\mathbb{H}}
\global\long\def\O{\mathcal{O}}
\global\long\def\E{\mathcal{E}}
\global\long\def\C{\mathbb{C}}
\global\long\def\D{\mathbb{\mathbf{D}}}
\global\long\def\DD{\overline{\D}}
\global\long\def\U{\mathfrak{U}}
\global\long\def\G{\mathbb{G}}
\global\long\def\O{\mathcal{O}}
\global\long\def\X{\mathfrak{X}}
\global\long\def\Y{\mathfrak{Y}}
\global\long\def\V{\mathbb{V}}
\global\long\def\be{^{\flat}}
\global\long\def\FF{\mathcal{F}}
\global\long\def\GG{\mathcal{G}}
\global\long\def\HH{\mathcal{H}}
\global\long\def\gS{\frak{G}}
\global\long\def\sS{\gS}
\global\long\def\bbA{\mathbb{A}}
\global\long\def\M{\mathcal{M}}
\global\long\def\T{\mathbf{T}}
\global\long\def\A{\mathbf{A}}
\global\long\def\R{\mathbb{R}}
\global\long\def\aa{\mathfrak{a}}
\global\long\def\bb{\mathfrak{b}}
\global\long\def\Mtor{{\bf M}^{{\rm tor}}}
\global\long\def\fil{{\rm Fil}}
\global\long\def\gr{{\rm Gr}}
\global\long\def\ff{\frak{f}}
\global\long\def\h{H_{{\rm dR}}^{1}}

\global\long\def\gL{\Lambda}
\global\long\def\gG{\Gamma}
\global\long\def\gO{\Omega}
\global\long\def\Om{\Omega}
\global\long\def\hh{\frak{h}}
\global\long\def\w{\mathbb{W}}

\global\long\def\gl{\lambda}
\global\long\def\gg{\gamma}
\global\long\def\gb{\beta}
\global\long\def\ga{\alpha}
\global\long\def\go{\omega}
\global\long\def\gd{\delta}
\global\long\def\e{\epsilon}
\global\long\def\s{\sigma}

\global\long\def\su{\subseteq}
\global\long\def\ot{\otimes}
\global\long\def\op{\oplus}
\global\long\def\b{\bullet}

\global\long\def\inv{^{\times}}
\global\long\def\ua{^{\ast}}
\global\long\def\da{_{\ast}}

\global\long\def\ad{^{{\rm ad}}}
\global\long\def\rig{^{\mathrm{rig}}}
\global\long\def\na{\nabla}
\global\long\def\dag{^{\dagger}}
\global\long\def\et{\mathrm{\acute{e}t}}
\global\long\def\an{^{\mathrm{an}}}
\global\long\def\can{C\an}
\global\long\def\An{\mathrm{An}}
\global\long\def\art{\mathrm{Art}}
\global\long\def\defi{\mathrm{Def}}
\global\long\def\tor{^{{\rm tor}}}
\global\long\def\an{^{{\rm an}}}
\global\long\def\red{_{{\rm red}}}
\global\long\def\ext{\mathrm{Ext}}
\global\long\def\Rext{\mathrm{R}\ext}

\global\long\def\vg{\underline{\gamma}}
\global\long\def\uo{\underline{\go}}
\global\long\def\oa{\overline{A}}
\global\long\def\ac{A_{\C_{K}}}
\global\long\def\uxi{\underline{\xi}}
\global\long\def\vp{\varpi}
\global\long\def\ud{\underline{\gd}}
\global\long\def\uX{\underline{X}}
\global\long\def\W{\mathcal{W}}
\global\long\def\an{^{\mathrm{an}}}
\global\long\def\k{\kappa}

\title{Modular sheaves of de Rham classes on Hilbert formal modular schemes
for unramified primes}

\author{Giacomo Graziani\thanks{giac.graz@gmail.com}}
\maketitle

\paragraph{Abstract}

We define formal vector bundles with marked sections on Hilbert modular
schemes and we show how to use them to construct modular sheaves with
an integrable meromorphic connection and a filtration which, in degree
0, gives to us a $p$-adic interpolation of the usual Hodge filtration.
We relate it with the sheaf of overconvergent Hilbert modular forms.

\tableofcontents{}

\newpage{}

\section{Introduction}

\subsubsection*{A quick historical recap}

The story of geometric modular forms started with the work of Katz
in \cite{ModIII}, aimed principally to create a unified geometric
framework for various notions of modular forms (e.g. Serre's and overconvergent
modular forms) of integral weight. The main idea was to construct
compactifications of suitable moduli spaces for elliptic curves and
use forms on the universal object to define a sheaf whose sections
are indeed modular forms.

Coleman, in his paper \cite{Col}, addressed a conjecture of Gouvea
predicting that every overconvergent modular form of integral weight
and sufficently small slope is classical and the main tool in his
research was the introduction of the sheaves 
\[
{\cal H}_{k}={\rm Sym}_{\O_{X}}^{k}\mathbf{R}^{1}\pi_{\ast}\Omega_{E/X}^{\b}\left\langle C\right\rangle 
\]
that is the degree $k$-th part of the symmetric algebra constructed
on the first algebraic de Rham cohomology sheaf with log poles at
the fibers above the cusps, where $\pi:E\to X$ is the generalized
universal elliptic curve, toghether with the natural filtration induced
by the Hodge filtration. Taking duals this definition is easily extended
to all classical weights $k\in\Z$ thus providing a triple $\left({\cal H}_{k},\nabla_{k},F^{\b}{\cal H}_{k}\right)$
for each $k\in\Z$, where $\nabla_{\kappa}$ is the Gauss-Manin connection
on ${\cal H}_{k}$. He then accomplishes his task by means of a careful
study of the cohomology of such triples.

In \cite{AI} the authors generalized the construction of Coleman
and defined $p$-adic families of de Rham cohomology classes, having
as a main motivation the extension of the construction of triple product
$p$-adic $L$-functions to the more general case of finite slope
families of modular forms. Roughly, the idea is to $p$-adically interpolate
the sheaves ${\cal H}_{k}$ where now $k$ is a $p$-adic weight,
that is a continuous character of $\Z_{p}\inv$. More precisely they
introduce special open subsets ${\cal W}_{I}$ of the adic weight
space and, using the theory of the canonical subgroup, define formal
models $\X_{r,I}$ of strict open neighbourhoods ${\cal X}_{r,I}$
of the ordinary locus in $X_{\Q_{p}}^{{\rm an}}\times{\cal W}_{I}$
and finite coverings $\IG_{n,r,I}\to\X_{r,I}$ over which the dual
of the canonical subgroup of level $n$ has a canonical generator.
Using these data they construct a subsheaf $H^{\#}$ of $H_{{\rm dR}}^{1}\left(\frak{E}/\IG_{n,r,I}\right)$
and a marked section (in the sense of \cite[Section 2]{AI}). The
machinery of formal vector bundles with marked sections provides a
sheaf $\mathbb{W}_{\kappa}$ in Banach modules over ${\cal X}_{r,I}$,
together with a natural filtration $F^{\b}\mathbb{W}_{\kappa}$ in
locally free coherent modules with the property that, for $k\in\Z$
(which means $k\left(x\right)=x^{k}$) we have an equality on the
rigid analytic space
\[
F^{k}\mathbb{W}_{k}={\cal H}_{k}
\]
and an integrable connection $\nabla$ induced by the Gauss-Manin
connection on $H_{{\rm dR}}^{1}$. The interest in this construction
is that $\nabla$ can be $p$-adically interpolated: working with
$q$-expansions, if $s$ is a continuous character of $\Z_{p}\inv$
over a complete $\Z_{p}$-algebra $R$, define $d^{s}:R\left[\left[q\right]\right]^{U=0}\to R\left[\left[q\right]\right]^{U=0}$
as
\[
d^{s}\left(\sum_{\begin{array}{c}
n\ge1\\
p\nmid n
\end{array}}a_{n}q^{n}\right)=\sum_{\begin{array}{c}
n\ge1\\
p\nmid n
\end{array}}s\left(n\right)a_{n}q^{n}.
\]
Then, given another continuous character $\kappa$ of $\Z_{p}\inv$,
under some mild assumptions on $\kappa$ and $s$, for every $\omega$
local section of $\mathbb{W}_{\kappa}$ they define
\[
\nabla_{\kappa}^{s}\left(\omega\right)\in\mathbb{W}_{\kappa+2s}
\]
such that on $q$-expansions
\begin{equation}
\nabla_{\kappa}^{s}\left(\omega\right)\left(q\right)=d^{s}\left(\omega\left(q\right)\right).\label{eq:iteration}
\end{equation}
This turns out to be the crucial point in the construction of the
triple product $p$-adic $L$-functions for forms of finite slope. 

\subsubsection*{Hilbert modular forms}

\emph{Hilbert modular forms} can be seen as a higher dimensional generalisation
of (elliptic) modular forms, where the role of ${\rm GL}_{2}\left(\Q\right)$
is played by ${\rm GL}_{2}\left(L\right)$ for a totally real number
field $L$. From the geometric point of view this means that we need
to consider the so-called \emph{abelian varieties with real multiplication
by $\O_{L}$} instead of ellitic curves. In \cite{TiXi} the authors
extended the work of Coleman in the case of Hilbert modular forms,
again with a classicity result in mind, and the purpose of this thesis
is to build up the technical apparatus needed to extend the work of
Andreatta-Iovita to the case of Hilbert modular forms. More precisely
fix a totally real number field $L$, say of degree $g$ over $\Q$,
and let $N\ge4$ and $p$ be two coprime natural numbers with $p$
prime. Finally let $\X$ be the ($p$-adic formal scheme associated
with a) smooth toroidal compactification of the moduli space of abelian
schemes with real multiplication by $\O_{L}$ and $\mu_{N}$-level
structure. The advantage of working with toroidal compactifications
lies in the fact that we have an universal semi-abelian object $\pi:\A\to\X$
with an action of $\O_{L}$. Weights, in this setting, are locally
analytic characters of the group $\left(\Z_{p}\ot_{\Z}\O_{L}\right)\inv$
(we can actually consider all of them at once taking a universal character
$\kappa$). 

First we consider a slight variation of the construction of formal
vector bundles with sections given in \cite{AI} to keep track of
the action of $\O_{L}$: pick a Galois closure $L^{{\rm Gal}}$ of
$L$ and let $d_{L}\in\Z$ be the discriminant of $L$, the main idea
in this construction is that, if $R$ is an $\O_{L^{{\rm Gal}}}\left[d_{L}^{-1}\right]$-algebra,
then we have a ring isomorphism
\[
\O_{L}\ot_{\Z}R\to\prod_{\sigma\in\frak{G}}R_{\sigma},
\]
given by $x\ot1\mapsto\left(\sigma\left(x\right)\right)_{\sigma}$,
where $\frak{G}$ is the set of embeddings $L\to L^{{\rm Gal}}$ and
$R_{\sigma}=R$. Therefore for every $\O_{L}\ot_{\Z}R$-module $M$
we have a canonical decomposition as $R$-modules 
\[
M=\prod_{\sigma\in\frak{G}}M\left(\sigma\right)
\]
and we can construct formal vector bundles with $\O_{L}$-action by
working for each $\sigma$ separatedly. 

Using the theory of canonical subgroups developped in \cite{AIP2}
we define formal models $\X_{r}$ of the overconvergent rigid analytic
neighborhoods and finite coverings $\IG_{n,r,I}\to\X_{r,I}$ parametrising
bases of the dual of the $n$-th canonical subgroup. Over these schemes
we have locally free coherent $\O_{\IG_{n,r,I}}\ot_{\Z}\O_{L}$-modules
$\Omega_{\A}\su H^{\#}$ with a common marked section, that allow
us, using the machinery explained above, to define sheaves of $\O_{\X_{r}}\ot_{\Z}\O_{L}$-modules
$\mathbb{W}_{\kappa}$ and $\mathfrak{w}^{\kappa}$ such that (see
Theorem \ref{thm:filtrazionesuX} for a precise statement) 
\begin{thm*}
$\,$
\begin{enumerate}
\item The sheaf $\mathbb{W}_{\kappa}$ comes with a natual increasing filtration
$F^{\b}\mathbb{W}_{\kappa}$ by locally free coherent $\O_{\X_{r}}\ot_{\Z}\O_{L}$-modules 
\item $\mathbb{W}_{\kappa}$ is isomorphic to the completed limit $\widehat{\varinjlim}F^{h}\mathbb{W}_{\kappa}$
and the graded pieces are
\[
{\rm Gr}^{h}F^{\b}\mathbb{W}_{\kappa}\simeq\frak{w}^{\kappa}\ot_{\O_{\X}}\uo_{\A}^{-2h};
\]
where $\uo_{\A}$ is pullback of the universal object $\A\to\X_{r}$
along the zero section;
\item $F^{0}\mathbb{W}_{\kappa}\simeq\frak{w}^{\kappa}$ and its sections
over the rigid analytic fibre ${\cal X}_{r}$ of $\X_{r}$ are the
overconvergent Hilbert modular forms as in \cite{TiXi}.
\end{enumerate}
\end{thm*}
Since the sheaf $\mathbb{W}_{\kappa}$ is constructed by means of
the first de Rham cohomology of the morphism $\pi$, the theory of
formal vector bundles with marked sections provides an integrable
connection $\nabla_{\kappa}$ over $\IG_{n,r,I}$, descending to a
meromorphic integrable connection on $\X_{r}$. However, on the analytic
space ${\cal X}_{r}$ we have (see Proposition \ref{prop:esisteconnessionesuW}
for a precise statement)
\begin{thm*}
On ${\cal X}_{r}$ the induced map on graded pieces
\[
\gr^{h}F^{\b}\nabla_{\kappa}:{\rm Gr}^{h}F^{\b}\mathbb{W}_{\kappa}\to{\rm Gr}^{h+1}F^{\b}\mathbb{W}_{\kappa}\widehat{\ot}_{\O_{{\cal X}_{r,I}}}\Omega_{{\cal X}_{r,I}}^{1}
\]
is well-defined and it is the composition of an isomorphism and the
product by an explicit element depending only on $\kappa$ and $h$.
\end{thm*}
Related results are obtained independently in \cite{Ayc}.

\subsection*{Future perspectives}

As said above, the goal of this thesis is to provide the technical
tools necessary to extend the work of Andreatta-Iovita \cite{AI}
to the case of Hilbert modular forms, hence let me briefly explain
two important applications.

\subsubsection*{De Rham cohomology and cusp forms}

Using the results in Section \ref{subsec:The-Gauss-Manin-connection-1}
it is possible to construct the de Rham complex
\[
\mathbb{W}_{\kappa}^{\b}:0\to\mathbb{W}_{\kappa}\to\mathbb{W}_{\kappa}\widehat{\ot}_{\O_{{\cal X}_{r}}}\Omega_{{\cal X}_{r}}^{1}\to\dots\to\mathbb{W}_{\kappa}\widehat{\ot}_{\O_{{\cal X}_{r}}}\Omega_{{\cal X}_{r}}^{g}\to0.
\]
\begin{itemize}
\item In the elliptic case \cite[Corollary 3.35, pag. 42]{AI} it is shown,
using sheaf cohomology arguments, that over an open subset ${\cal U}$
of the weight space given by removing a finite number of \emph{classical}
points there is an isomorphism
\[
H_{{\rm dR}}^{1}\left({\cal X}_{r},\mathbb{W}_{k}^{\b}\right)^{\left(h\right)}\ot\O\left({\cal U}\right)\simeq H^{0}\left({\cal X}_{r},\frak{w}^{k+2}\right)^{\left(h\right)}\ot\O\left({\cal U}\right)
\]
where ${\cal U}$ depends on $h$ and the slope decomposition is taken
with respect to the action of $U_{p}$. 
\item In the Hilbert case, in \cite[Theorem 3.5, pag. 95]{TiXi}it shown
that there is a sujective map 
\[
S_{{\bf k}}^{\dagger}\to H_{{\rm rig}}^{g}\left({\cal X}_{r}^{{\rm ord}},D,{\cal F}_{{\bf k}}\right)
\]
whose kernel can be described by means of a certain differential operator
$\Theta$, where ${\bf k}$ is a cohomological classical multiweight,
$S^{\dagger}$ denotes the module of overconvergent cusp forms, $D$
is the boundary divisor for a fixed toroidal compactification and
${\cal F}_{{\bf k}}$ is a sheaf similat to our $\mathbb{W}_{{\bf k}}$.
\end{itemize}
These two results suggest that it will be possible to relate the de
Rham cohomology $H_{{\rm dR}}^{g}\left({\cal X}_{r},\mathbb{W}_{k}\right)$,
or some finite slope cusp submodule thereof, with the space of Hilbert
cusp forms. The initial step for this task will be, following \cite{TiXi},
the study of the dual BGG complex of $\mathbb{W}_{\kappa}$ in the
spirit of \cite[Section 2.15, pag. 90]{TiXi}.

\subsubsection*{Iteration of the Gauss-Manin connection and construction of triple
product $p$-adic $L$-functions for finite slope Hilbert modular
forms for unramified primes}

The main result in \cite{AI}, as explained above, is the construction
of triple product $p$-adic $L$-functions in the more general case
of finite slopes elliptic modular forms instead of ordinary ones,
and the extension of this result to the case of Hilbert modular forms
is the natural application of the work carried out in this thesis.
The main technical obstacle is the construction of iterations of the
Gauss-Manin connection introduced in Section \ref{subsec:The-Gauss-Manin-connection-1},
that is of sections as in (\ref{eq:iteration}) under suitable conditions
on the weights $s$ and $k$. 

This work will require the definition and the study of the notion
of \emph{nearly overconvergent Hilbert modular forms} rely either
on computations using $q$-expansions (\cite{AI}) or in terms of
Serre-Tate coordinates (\cite{Fan} or a recent unpublished work of
S. Molina).

\section{Preliminaries}
\begin{notation*}
Let $L$ be a totally real number field of degree $g$ over $\Q$
and fix a prime $p$ which is unramified in $L$. Let\foreignlanguage{english}{
\[
p\O_{L}=\p_{1}\dots\p_{d}
\]
and}
\begin{equation}
\O_{L}\ot_{\Z}\Z_{p}=\prod\O_{F_{i}},\label{eq:decomposizioneocmpletamentoO_L}
\end{equation}
then $F_{i}$ is an unramified extension of $\Q_{p}$. 
\end{notation*}

\subsection{The weight space}

This chapter is essentially borrowed from \cite[Section 2]{AIP1}.

Let $\T={\rm Res}_{\O_{L}/\Z}\G_{m}$ and denote with $\gL=\gL_{L}$
the completed group algebra
\[
\gL=\Z_{p}\left[\left[\T\left(\Z_{p}\right)\right]\right],
\]
with $\k^{{\rm u}}:\T\left(\Z_{p}\right)\to\gL\inv$ the natural inclusion,
which we will refer to as the universal character and 
\[
\gL^{0}=\Z_{p}\left[\left[\T\left(\Z_{p}\right)_{{\rm tf}}\right]\right]
\]
be the completed group algebra corresponding to the torsion-free quotient
$\T\left(\Z_{p}\right)_{{\rm tf}}$, so that we can consider the composition
\[
\k:\T\left(\Z_{p}\right)\overset{\k^{{\rm u}}}{\longrightarrow}\gL\inv\to\left(\gL^{0}\right)\inv.
\]
\begin{lem}
The ring $\gL^{0}$ is a regular local ring of Krull dimension $g+1$.
Let $\gg_{1},\dots,\gg_{g}$ be any topological basis of $\T\left(\Z_{p}\right)_{{\rm tf}}$
and let $\m\su\gL^{0}$ be the ideal generated by $p,\gg_{1}-1,\dots,\gg_{g}-1$,
then $\m$ is the maximal ideal of $\gL^{0}$ and $\gL^{0}$ is complete
with respect to the $\m$-adic topology.
\end{lem}

Let $\frak{W}=\spf\left(\gL\right)$ and $\frak{W}^{0}=\spf\left(\gL^{0}\right)$.
We consider the admissible (formal) blow-up $t:{\rm Bl}_{\m}\frak{W}\to\frak{W}$
along the ideal $\m$. In the same way we define $t^{0}:{\rm Bl}_{\m}\frak{W}^{0}\to\frak{W}^{0}$.
We have a finite locally free natural map ${\rm Bl}_{\m}\frak{W}\to{\rm Bl}_{\m}\frak{W}^{0}$.

We will work with the adic spaces associated with such formal schemes. 
\begin{note}
\label{Note:punti analitici}Given an affinoid adic space $X=\spa\left(A,A^{+}\right)$,
a point $x\in X$ is called analytic if ${\rm Supp}\left(x\right)\in\spec\left(A\right)$
is not open. If $X$ is not affine, then $x\in X$ is called analytic
if there exists an open neighborhood $x\in U$ such that $\O_{X}\left(U\right)$
contains a topologically nilpotent unit. We will denote with $X^{{\rm an}}\su X$
the open subspace of analytic points. Note that if $\X=\spf\left(A\right)$
is a noetherian affine formal scheme where $A$ complete with respect
to the topology induced by an ideal $\aa$. Let and let $t:{\rm Bl}_{\aa}\X\to\X$
be the admissible formal blow-up along $\aa$. Then $t$ is an adic
morphism and the corresponding map
\[
t^{{\rm an}}:\left({\rm Bl}_{\aa}\X\right)^{{\rm an}}\to\X\an
\]
on the analytic locus is an open embedding. We ay that an adic space
$X$ is analytic if $X=X^{{\rm an}}$
\end{note}

\begin{defn}
Define ${\cal W}^{0}=\left(\frak{W}^{0}\right)\an$ and for every
$\ga\in\m$ let $\W_{\ga}^{0}=\left(\frak{W}_{\ga}^{0}\right)\an$.

Let $\frac{r}{s}\in\Q_{>0}$ be a reduced fraction, define the following
subsets of ${\cal W}^{0}$:
\begin{itemize}
\item $\W_{\le\frac{r}{s}}^{0}=\left\{ x\in\W^{0}\,\mid\,\left|\ga^{r}\right|_{x}\le\left|p^{s}\right|_{x}\neq0\quad\forall\ga\in\m\right\} $
\item $\W_{\ge\frac{r}{s}}^{0}=\left\{ x\in\W^{0}\,\mid\,0\neq\left|\ga^{r}\right|_{x}\ge\left|p^{s}\right|_{x}\quad\forall\ga\in\m\right\} $
\item $\W_{\ge0}^{0}=\W_{\le\infty}^{0}=\W^{0}$
\item for $a,b\in\Q_{>0}\cup\left\{ \infty\right\} $ and $I=\left[a,b\right]$
set $\W_{I}^{0}=\W_{\le b}^{0}\cap\W_{\ge a}^{0}$
\item for $\ga\in\m$ we let $\W_{\ga,I}^{0}=\W_{I}^{0}\cap\W_{\ga}^{0}$.
\end{itemize}
We introduce formal models for these spaces: fix an interval $I=\left[a,b\right]\su\Q_{>0}\cup\left\{ \infty\right\} $.
For $\ga\in\m$ let $\frak{\frak{W}}_{\ga}^{0}=\spf\left(B_{\ga}\right)$,
set $B_{\ga,I}^{0}=H^{0}\left(\W_{\ga,I}^{0},\O_{\W_{\ga,I}^{0}}^{+}\right)$
and $\frak{W}_{\ga,I}^{0}=\spf\left(B_{\ga,I}^{0}\right)$. The analytic
fibre of $\frak{W}_{\ga,I}^{0}$ is $\W_{\ga,I}^{0}$ and that they
give an affine cover of a locally noetherian formal scheme $\frak{W}_{I}^{0}$
whose analytic fibre is $\W_{I}^{0}$. We let
\[
\k_{I}:T\left(\Z_{p}\right)\to\left(\O_{\W_{I}}^{+}\right)\inv
\]
be the natural map.
\end{defn}

\begin{prop}[Analyticity of the universal character]
\label{prop:analiticitadecarattereuniversale}Let $n,m\ge0$ be integers
and let $I\su\left[0,p^{n}\right]\cap\Q$ be a closed interval. Set
\[
\epsilon=\begin{cases}
1 & p\neq2\\
3 & \mbox{otherwise}
\end{cases},
\]
then $\k$ induces a pairing
\[
\k_{{\rm p}}:\frak{W}_{I}^{0}\times\T\left(\Z_{p}\right)\cdot\left(1+p^{n+\epsilon}\cdot\Res_{\O_{L}/\Z}\left(\G_{a}\right)\right)\to\G_{m}
\]
on the category of $p$-adic formal schemes on $\frak{W}_{I}^{0}$
that restricts to
\[
\k_{{\rm p}}:\frak{W}_{I}^{0}\times\left(1+p^{n+m+\epsilon}\cdot\Res_{\O_{L}/\Z}\left(\G_{a}\right)\right)\to1+qp^{m}\G_{a}.
\]
\end{prop}

\begin{note}
\label{note:esisteelementou_I}Note that, for a $p$-adically complete
and separated ring $A$ with ideal of definiton $I$, for $n$ large
enough the exponential power series is convergent on $I^{n}$ (indeed
one can take $n$ such that $I^{n}\su pA$) , it follows that there
exists an element $u_{I}\in p^{1-n}$ such that $k_{I}\left(t\right)=\exp\left(u_{I}\cdot\log\left(t\right)\right)$
for $t\in1+p^{n}\O_{L}\ot_{\Z}A$.
\end{note}

\subsection{Formal $\protect\O_{F}$-module bundles with marked sections}

In this section we consider additional linear structures on formal
vector bundles as introduced in \cite{AI}. Let $\Q_{p}\su F$ be
a finite unramified and $\sS$ be the set of embeddings $\s:F\to F$.
Let moreover $\X\to S=\spf\left(\O_{F}\right)$ be an admissible scheme
with an invertible ideal of definition ${\cal I}=p^{n}\O_{\X}$. We
will denote with $\overline{\X}$ the reduction modulo ${\cal I}$.

Given an $\O_{F}\ot_{\Z_{p}}\O_{\X}$-module $\E$ and $\s\in\gS$
we denote with $\E\left(\s\right)$ the sub $\O_{\X}$-module on which
$\O_{F}$ acts via $\s$. Since $F$ is unramified over $\Q_{p}$
we have a ring isomorphism
\begin{align*}
\O_{F}\ot_{\Z_{p}}\O_{\X} & \simeq\prod_{\s\in\gS}\O_{\X,\s}\\
x\ot1 & \mapsto\left(\s\left(x\right)\right)_{\s}
\end{align*}
and a decomposition
\[
\E\simeq\bigoplus_{\s\in\gS}\E\left(\s\right).
\]
\begin{thm}
\label{thm:sezioni marcate}Let $\E$ be a coherent locally free $\O_{F}\ot_{\Z_{p}}\O_{\X}$-module,
and let $s\in H^{0}\left(\overline{\X},\overline{\E}\right)$ be a
non-zero element such that the sub-$\O_{F}\ot_{\Z_{p}}\O_{\X}$-module
generated by the lifts of $s$ is locally a direct summand. Then
\begin{enumerate}
\item there exists an admissible morphism
\[
\pi:\V_{\O_{F}}\left(\E\right)\to\X
\]
representing the functor in $\O_{F}\ot_{\Z_{p}}\O_{\X}$-modules
\[
\V_{\O_{F}}\left(\E\right):\left(t:\frak{Y}/\X\right)\mapsto{\rm Hom}_{\O_{\frak{Y}}}\left(t^{\ast}\E,\O_{\frak{Y}}\right)={\rm Hom}_{\O_{\X}}\left(\E,t_{\ast}\O_{\frak{Y}}\right),
\]
For $\spf\left(R\right)\su\X$ we have
\[
\pi_{\ast}\O_{\V_{\O_{F}}\left(\E\right)}\left(\spf\left(R\right)\right)=\widehat{\bigotimes_{\s\in\sS}}{\rm Sym}_{R}\left(\E_{|R}\left(\s\right)\right).
\]
Moreover we have a commutative diagram
\[
\xymatrix{\V_{\O_{F}}\left(\E\right)\ar[rr]^{\pi}\ar[dr] &  & \X\\
 & \X\times_{\spf\left(\Z_{p}\right)}\spf\left(\O_{F}\right)\ar[ur]
}
\]
\item the functor 
\[
\V_{0}\left(\E,s\right):\left(t:\frak{Y}/\X\right)\mapsto\left\{ f\in\V_{\O_{F}}\left(\E\right)\left(t:\frak{Y}/\X\right)\,\mid\,\left(f\mod p^{n}\right)\left(t^{\ast}s\right)\in\underline{\left(\O_{F}/p^{n}\right)\inv}\right\} 
\]
is represented by an open subset of an admissible blow-up of $\V_{\O_{F}}\left(\E\right)$.
\end{enumerate}
\end{thm}

We refer to the element $s$ of Theorem \ref{thm:sezioni marcate}
as a \emph{marked section} of $\E$ and to the pair $\left(\E,s\right)$
as an ${\rm MS}_{\O_{F}}$-datum. L\foreignlanguage{english}{et $S\su\O_{F}\inv$
be a faithful set of lifts of the elements of $\left(\O_{F}/p^{n}\right)\inv$
and fix $u\in S$, let moreover $\left\{ e_{1,\s},\dots,e_{m,\s}\right\} $
be a basis for $\E_{|R}\left(\s\right)$ for $\spf\left(R\right)\su\X$.
Then
\begin{align*}
\pi_{\ast}\O_{\V_{\O_{F}}\left(\E\right)}\left(\spf\left(R\right)\right) & =\widehat{\bigotimes_{\s\in\gS}}\pi_{\ast}\O_{\V_{\O_{F}}\left(\E\right)_{\s}}\left(\spf\left(R\right)\right)\\
 & =\widehat{\bigotimes_{\s\in\gS}}R\left\langle X_{1,\s},\dots,X_{m,\s}\right\rangle .
\end{align*}
}The formal scheme representing $\V_{0}\left(\E,s\right)$ then is
the disjoint union of the schemes $\V_{0}^{u}\left(\E,s\right)$ index
by $u\in S$, where
\[
\V_{0}^{u}\left(\E,s\right)\left(t:\frak{Y}/\X\right)=\left\{ f\in\V_{0}\left(\E,s\right)\left(t:\frak{Y}/\X\right)\,\mid\,\left(f\mod p^{n}\right)\left(t^{\ast}s\right)\in\overline{u}\right\} .
\]
Denote with\foreignlanguage{english}{ $\frak{f}:\V_{0}\left(\E,s\right)\to\X$
the composition 
\[
\V_{0}\left(\E,s\right)\overset{\xi}{\longrightarrow}\V_{\O_{F}}\left(\E\right)\overset{\pi}{\longrightarrow}\X,
\]
then the map
\[
\xi^{\#}:\pi_{\ast}\O_{\V_{\O_{F}}\left(\E\right)}\to\frak{f}_{\ast}\O_{\V_{0}\left(\E,s\right)}
\]
is locally
\begin{align*}
R\left\langle X_{1,\s},X_{2,\s},\dots,X_{m,\s}\,\middle|\,\s\in\sS\right\rangle  & \to\prod_{u\in S}R\left\langle Z_{\s,u},X_{2,\s,u},\dots,X_{m,\s,u}\,\middle|\,\s\in\sS\right\rangle \\
X_{i,\s} & \mapsto\begin{cases}
\left(p^{n}Z_{\s,u}+\s\left(u\right)\right)_{u\in S} & i=1\\
\left(X_{i,\s,u}\right)_{u\in S} & i\neq1
\end{cases}
\end{align*}
}
\begin{note}
\label{note:filtrazionesulfibrato}Let $\E$ be a coherent locally
free $\O_{F}\ot_{\Z_{p}}\O_{\X}$-module and consider the formal scheme
$\pi:\V_{\O_{F}}\left(\E\right)\to\X$. For $\spf\left(R\right)\su\X$
such that $\E_{|\spf\left(R\right)}$ is free, say with a basis $\left\{ e_{i,\s}\,\mid\,\begin{array}{c}
i=1,\dots,n\\
\s\in\sS
\end{array}\right\} $ we define\footnote{There many other possibilities leading to the same statements, for
example i\foreignlanguage{english}{f $\mathbf{k}=\left(k_{\s}\right)_{\s}\in\mathbb{N}^{\sS}$
set $\mid\mathbf{k}\mid=\sum k_{i}$, then one could define
\[
F^{\mathbf{k}}\pi_{\ast}\O_{\V_{\O_{F}}\left(\E\right)}\left(\spf\left(R\right)\right)=\bigotimes_{\begin{array}{c}
\s\in\sS\\
\mid\mathbf{k}\mid=h
\end{array}}R\left[X_{1,\s},\dots,X_{n,\s}\right]_{\le k_{\s}}
\]
However we picked the one indexed by a totally ordered set.}}
\[
F^{h}\pi_{\ast}\O_{\V_{\O_{F}}\left(\E\right)}\left(\spf\left(R\right)\right)=\bigotimes_{\s\in\sS}R\left[X_{1,\s},\dots,X_{n,\s}\right]_{\le h}.
\]
We see that picking another basis of $\E\left(\s\right)$ doesn't
affect the $R$-module $R\left[X_{1,\s},\dots,X_{n,\s}\right]_{\le h}$,
therefore the modules $F^{h}\pi_{\ast}\O_{\V_{\O_{F}}\left(\E\right)}$
define an increasing filtration of $\pi_{\ast}\O_{\V_{\O_{F}}\left(\E\right)}$,
moreover we have
\[
\pi_{\ast}\O_{\V_{\O_{F}}\left(\E\right)}=\widehat{\varinjlim_{h}}F^{h}\pi_{\ast}\O_{\V_{\O_{F}}\left(\E\right)}
\]
and
\[
\gr^{h}F^{\b}\pi_{\ast}\O_{\V_{\O_{F}}\left(\E\right)}=\bigotimes_{\s\in\sS}R\left[X_{1,\s},\dots,X_{n,\s}\right]_{h}.
\]
\foreignlanguage{english}{Suppose that for every $\s$ we have $s_{\s}=e_{1,\s}$,
let $\FF\su\E$ be the invertible local direct summand of lifts of
the marked section $s$. Note that $\left(\FF,s\right)$ is also an
${\rm MS}_{\O_{F}}$-datum and locally we have
\[
\FF_{|\spf\left(R\right)}=\left\langle e_{1,\s}\,\mid\,\s\in\sS\right\rangle .
\]
In this case we have a natural increasing filtration $\fil^{h}\frak{f}_{\ast}\O_{\V_{0}\left(\E,s\right)}$
for $h\ge0$ such that}
\end{note}

\begin{enumerate}
\item $\frak{f}_{\ast}\O_{\V_{0}\left(\E,s\right)}=\widehat{\varinjlim_{h}}F^{h}\frak{f}_{\ast}\O_{\V_{0}\left(\E,s\right)};$
\item $\gr^{h}F^{\b}\frak{f}_{\ast}\O_{\V_{0}\left(\E,s\right)}=\frak{f}_{\ast}\O_{\V_{0}\left({\cal F},s\right)}\ot_{\O_{\X}}\gr^{h}F^{\b}\pi_{\ast}\O_{\V_{\O_{F}}\left(\frac{\E}{\FF}\right)}.$
\end{enumerate}
Indeed it follows from the functorial description that
\[
\V_{0}\left(\E,s\right)=\V_{0}\left(\FF,s\right)\times_{\V_{\O_{F}}\left(\FF\right)}\V_{\O_{F}}\left(\E\right),
\]
then
\begin{align*}
\frak{f}_{\ast}\O_{\V_{0}\left(\E,s\right)} & =\frak{f}_{\ast}\O_{\V_{0}\left(\FF,s\right)}\widehat{\otimes}_{\pi_{\ast}\O_{\V_{\O_{F}}\left(\FF\right)}}\pi_{\ast}\O_{\V_{\O_{F}}\left(\E\right)}\\
 & =\frak{f}_{\ast}\O_{\V_{0}\left(\FF,s\right)}\widehat{\otimes}_{\O_{\X}}\pi_{\ast}\O_{\V_{\O_{F}}\left(\frac{\E}{\FF}\right)}.
\end{align*}
and we get
\[
F^{h}\frak{f}_{\ast}\O_{\V_{0}\left(\E,s\right)}=\frak{f}_{\ast}\O_{\V_{0}\left(\FF,s\right)}\ot_{\O_{\X}}F^{h}\pi_{\ast}\O_{\V_{\O_{F}}\left(\frac{\E}{\FF}\right)}
\]
 Locally we have
\[
\fil^{h}\frak{f}_{\ast}\O_{\V_{0}\left(\E,s\right)}\left(\spf\left(R\right)\right)=\prod_{u\in S}R\left\langle Z_{\s,u}\,\mid\,\s\in\sS\right\rangle \left[X_{i,\s,u}\,\middle|\,\begin{array}{c}
i=2,\dots,n\\
\s\in\sS
\end{array}\right]_{\le h}.
\]

\subsubsection{Connections on the sheaf of functions \label{subsec:Connections-on-the}}

In this section we refer to \cite[Section 2]{BeOg} for Grothendieck's
formalism on connections. Let $X\to S$ be a scheme and let ${\cal A}={\cal A}_{X/S}$
be the $\O_{X}$-algebra $\O_{X}\op\Omega_{X/S}^{1}$ where the product
is (locally) given by 
\[
\left(s,\omega\right)\cdot\left(t,\tau\right)=\left(st,s\tau+t\omega\right).
\]
Note that $\Omega_{X/S}^{1}\su{\cal A}$ is a square-zero ideal. We
have two obvious ring maps $j_{1},j_{2}:\O_{X}\to{\cal A}$ given
by
\[
j_{1}\left(s\right)=\left(s,0\right)\quad\quad j_{2}\left(s\right)=\left(s,ds\right)
\]
and the natural quotient $\Delta:{\cal A}\to\O_{X}$. 

Let now $\X\to S$ be a formal scheme and $\left(\E,s\right)$ be
an ${\rm MS}_{\O_{F}}$-datum with $\nabla:\E\to\E\widehat{\ot}_{\O_{\X}}\Omega_{\X/S}^{1}$
an integrable connection. We say that it is an ${\rm MS}_{\O_{F}}$-connection
if
\begin{itemize}
\item it is compatible with the $\O_{F}$-structure, that is
\[
\nabla\left(\E\left(\s\right)\right)\su\E\left(\s\right)\widehat{\ot}_{\O_{\X}}\Omega_{\X/S}^{1}.
\]
In this case the restriction $\nabla_{\s}=\nabla_{|\E\left(\s\right)}$
is again an integrable connection and
\item $s$ is horizontal for $\overline{\nabla}$, that is $\overline{\nabla}\left(s\right)=0$.
\end{itemize}
In this case for every $\s$ we have an ${\cal A}$-linear isomoprhism
\[
\epsilon_{\s}:\E\left(\s\right)\ot_{j_{2}}{\cal A}\to\E\left(\s\right)\ot_{j_{1}}{\cal A}
\]
such that $\overline{\epsilon_{\s}}\left(s_{\s}\ot1\right)=s_{\s}\ot1$
and $\epsilon_{\s}\ot_{\Delta}\O_{\X}={\rm id}_{\E\left(\s\right)}$
plus a cocycle condition translating integrability. In particular
$\epsilon_{\s}$ is an isomorphism of ${\rm MS}_{\O_{F}}$-data over
${\cal A}$ and by functoriality we have an isomorphism of $\underline{\spf}_{\O_{\X}}\left({\cal A}\right)$-formal
schemes
\[
\V_{0}\left(\E,s\right)\times_{j_{2}}\underline{\spf}_{\O_{\X}}\left({\cal A}\right)\to\V_{0}\left(\E,s\right)\times_{j_{1}}\underline{\spf}_{\O_{\X}}\left({\cal A}\right)
\]
giving an isomorphism
\[
\frak{f}_{\ast}\O_{\V_{0}\left(\E,s\right)}\otimes_{\O_{\X}}{\cal A}\overset{\epsilon_{0}}{\longrightarrow}\frak{f}_{\ast}\O_{\V_{0}\left(\E,s\right)}\otimes_{\O_{\X}}{\cal A}.
\]
In view again of Grothendieck's formalism we have a commutative diagram
where the vertical arrows are integrable connections
\[
\xymatrix{\E\ar[r]\ar[d]_{\nabla} & \pi_{\ast}\O_{\V_{\O_{F}}\left(\E\right)}\ar[r]\ar[d]_{\nabla_{\O_{F}}} & \frak{f}_{\ast}\O_{\V_{0}\left(\E,s\right)}\ar[d]^{\nabla_{0}}\\
\E\widehat{\ot}_{\O_{\X}}\Omega_{\X/S}^{1}\ar[r] & \pi_{\ast}\O_{\V_{\O_{F}}\left(\E\right)}\widehat{\ot}_{\O_{\X}}\Omega_{\X/S}^{1}\ar[r] & \frak{f}_{\ast}\O_{\V_{0}\left(\E,s\right)}\widehat{\ot}_{\O_{\X}}\Omega_{\X/S}^{1}
}
\]
This allows to give a local description of the connections. If $\spf\left(R\right)\su\X$
is such that $\E_{|\spf\left(R\right)}$ is free with $R$-basis
\[
\left\{ e_{1,\s},\dots,e_{n,\s}\,\mid\,\s\in\sS\right\} \quad\mbox{with}\quad\overline{e_{1,\s}}=s_{\s},
\]
suppose 
\[
\nabla\left(e_{i,\s}\right)=\begin{cases}
\sum_{j}\ga r_{1,j,\s}e_{j,\s}\ot\omega_{1,j,\s} & i=1\\
\sum_{j}r_{i,j,\s}e_{j,\s}\ot\omega_{i,j,\s} & {\rm otherwise}.
\end{cases},
\]
then 
\[
\nabla_{\O_{F}}\left(X_{i,\s}\right)=\begin{cases}
\sum_{j}\ga r_{1,j,\s}X_{j,\s}\ot\omega_{1,j,\s} & i=1\\
\sum_{j}r_{i,j,\s}X_{j,\s}\ot\omega_{i,j,\s} & {\rm otherwise}.
\end{cases}
\]
We have $\nabla_{0}\left(\ga Z_{\s,u}\right)=\nabla_{\O_{F}}\left(X_{1,\s}\right)$
so that
\[
\nabla_{0}\left(Z_{\s,u}\right)=\sum_{j}r_{1,j,\s}X_{j,\s,u}\ot\omega_{1,j,\s}\quad\mbox{for every \ensuremath{\s\in\sS}. }
\]

Summing up we have, in view of the local descriptions and Leibniz'
rule
\begin{prop}
Let $\left(\E,s\right)$ be an ${\rm MS}_{\O_{F}}$-datum over $\X$
and let $\nabla:\E\to\E\widehat{\otimes}_{\O_{\X}}\Omega_{\X/S}^{1}$
be an integrable ${\rm MS}_{\O_{F}}$-connection. Then the connection
$\nabla_{0}$ on $\frak{f}_{\ast}\O_{\V_{0}\left(\E,s\right)}$ is
integrable, satisfies Griffith's transversality, that is
\[
\nabla_{0}\left(\fil^{h}\frak{f}_{\ast}\O_{\V_{0}\left(\E,s\right)}\right)\su\fil^{h+1}\frak{f}_{\ast}\O_{\V_{0}\left(\E,s\right)}\widehat{\otimes}_{\O_{\X}}\Omega_{\X/S}^{1}
\]
and the induced map
\[
\gr^{h}\left(\nabla_{0}\right):\gr^{h}\left(\frak{f}_{\ast}\O_{\V_{0}\left(\E,s\right)}\right)\to\gr^{h+1}\left(\frak{f}_{\ast}\O_{\V_{0}\left(\E,s\right)}\right)\widehat{\otimes}_{\O_{\X}}\Omega_{\X/S}^{1}
\]
is $\O_{\X}$-linear.
\end{prop}


\section{The Igusa tower}

We give here a brief account of the construction of the partial Igusa
tower over a Hilbert modular formal scheme as described in \cite[Section 3]{AIP1}.

Let $N\ge4$ be an integer and $p$ a prime number not dividing $N$.
Let ${\cal D}_{L}$ be the different ideal of $L$ over $\Q$. Fix
moreover a fractional ideal $\c\su L$ with cone of totally positive
elements $\c^{+}$. Then there exists a scheme $M_{N}\left(\c\right)\to\spec\left(\Z_{\left(p\right)}\right)$
that classifies quadruples $\left(A,\iota,\Psi,\gl\right)$ such that
$A\to S\to\spec\left(\Z_{\left(p\right)}\right)$ is an abelian scheme
with real multiplication by $\O_{L}$ induced by the ring embedding
$\iota:\O_{L}\to{\rm End}_{S}\left(A\right)$, $\Psi:\mu_{N}\ot{\cal D}_{L}^{-1}\to A$
is a closed immersion as $\O_{L}$-modules over $S$. Finally, if
${\cal M}_{A}\su{\rm Hom}_{\O_{L}}\left(A,A^{\lor}\right)$ is the
sheaf of symmetric $\O_{L}$-linear morphisms and ${\cal M}_{A}^{+}\su{\cal M}_{A}$
is the subset of polarisations, then $\gl:\left({\cal M}_{A},{\cal M}_{A}^{+}\right)\simeq\left(\c,\c^{+}\right)$
is an isomorphism of étale sheaves. We require that the natural map
$A\ot_{\O_{L}}\c\to A^{\lor}$ is an isomorphism. Let $M\to\spec\left(\Z_{\left(p\right)}\right)$
be a fixed toroidal compactification, it comes with a universal semiabelian
scheme $A\to X$. We let $\A\to\frak{M}\to\spf\left(\Z_{p}\right)$
denote the associated $p$-adic formal schemes. If $e:\frak{M}\to\A$
is the unit section, denote
\[
\uo_{\A}=e^{\ast}\Omega_{\A/\X}^{1},
\]
it is an invertible $\O_{\X}\ot_{\Z}\O_{L}$-module. 

We let $\Hdg\su\O_{\frak{M}}$ be the Hodge ideal, it is defined as
the ideal of lifts of the Hasse invariant. 

\subsection{Canonical subgroups\label{subsec:Canonical-subgroups}}

Fix a $\Z_{p}$-algebra $A_{0}$ and suppose it is an integral domain
and the $\ga$-adic completion of a $\Z_{p}$-algebra of finite type,
where $\ga\in A_{0}\backslash\left\{ 0\right\} $ has $p\in\ga A_{0}$.
Let $\Y\to\spf\left(A_{0}\right)$ be the formal base change of $\frak{M}$,
then for $r\in\mathbb{N}$, the formal scheme
\[
\Y_{r}\to\Y
\]
is defined by requiring that the ideal generated by $\ga$ is divisible
by $\Hdg^{p^{r+1}}$. See \cite[Section 3.2]{AIP1} for a precise
definition. 
\begin{prop}
\label{prop:propriet=0000E0sottogruppocanonico}Assume that $p\in\ga^{p^{k}}A_{0}$.
Then for every integer $1\le n\le r+k$ one has a canonical subgroup
scheme $H_{n}$ of $\A\left[p^{n}\right]$ over $\Y_{r}$ and $H_{n}$
modulo $p\Hdg^{-\frac{p^{n}-1}{p-1}}$lifts the kernel of the $n$-th
power Frobenius. Moreover $H_{n}$ is finite flat and locally of rank
$p^{ng}$, it is stable under the action of $\O_{L}$ and its Cartier
dual $H_{n}^{D}$ is étale-locally over $A_{0}\left[\ga^{-1}\right]$
isomorphic to $\O_{L}/p^{n}$ as $\O_{L}$-modules.
\end{prop}

Let ${\cal S}=\spa\left(A\left[\ga^{-1}\right],A^{+}\right)$, where
$A$ is an $\ga$-adically complete $A_{0}$-algebra topologically
of finite type and $A^{+}\su A\left[\ga^{-1}\right]$ is the normalisation
of $A$. Denote with $\Y_{r}^{{\rm an}}$ the analytic adic space
associated with $\Y_{r}$. Let $p\in\ga^{p^{k}}A$ and $1\le n\le r+k$,
then it follows from Proposition \ref{prop:propriet=0000E0sottogruppocanonico}
that we have an $\O_{L}$-module object $H_{n}$ with $p^{n}$-torsion
on $\Y{}_{r}\an$ which admits therefore a decomposition
\[
H_{n}\simeq\prod_{i=1}^{d}H_{n}^{\left(i\right)}
\]
in view of (\ref{eq:decomposizioneocmpletamentoO_L}). Then $H_{n}^{\left(i\right),D}$
is étale-locally isomorphic to $\underline{\left(\O_{F_{i}}/p^{n}\right)}$.
We can consider the $\underline{\left(\O_{F_{i}}/p^{n}\right)\inv}$-torsor
$\underline{{\rm Isom}}\left(\underline{\left(\O_{F_{i}}/p^{n}\right)},H_{n}^{\left(i\right),D}\right)$
on $\Y{}_{r}\an$ which is representable in view of \cite[Theorem 4.3.(a), pag. 121]{Mil1}.
Call $h^{\left(i\right),n}:{\cal I{\cal G}}_{n,r}^{\left(i\right)}\to\Y{}_{r}\an$
the resulting object\foreignlanguage{english}{: it }is finite étale
and Galois with group $\left(\O_{F_{i}}/p^{n}\right)\inv$. Taking
the normalisation we have a finite morphism 
\[
\frak{h^{\left(i\right),n}}:\IG_{n,r}^{\left(i\right)}\to\Y{}_{r}.
\]
\begin{rem}
\label{rem:ramificazioneIgusa}It follows from \cite[Corollary 3.5]{AIP1}
that, if $\spf\left(R\right)\su\Y_{r}$ is such that $\Hdg$ is generated
by $\tilde{{\rm Ha}}$, then for every $0\le n\le r+k$ there exist
elements $c_{0}=1$ and $c_{n}\in\tilde{{\rm Ha}}^{-\frac{p^{n}-1}{p-1}}\O_{\IG_{n,r}}\left(\spf\left(R\right)\right)$
such that 
\[
{\rm Tr}_{\IG_{n,r}/\IG_{n-1,r}}\left(c_{n}\right)=c_{n-1}.
\]

Let $\gb_{n}$ denote the ideal $p^{n}\Hdg^{-\frac{p^{n}-1}{p-1}}\su\O_{\IG_{n}^{\left(i\right)}}$.
This makes sense in view of \cite[Lemma 4.2, pag. 16]{AIP1}. Define
\[
\IG_{n,r}=\prod_{i=1}^{d}\IG_{n,r}^{\left(i\right)}\overset{\frak{h}^{n}}{\longrightarrow}\Y_{r}.
\]
\end{rem}


\subsection{A digression on polarisations }

Here we follows \cite[Section 3]{AnGo1}. Let $A_{/S}$ be an abelian
scheme with real multiplication by $\O_{L}$ and assume that it satisfies
condition (DP), that is, the natural map
\begin{align*}
\phi_{A}:A\ot_{\O_{L}}\mathcal{M}_{A} & \to A^{\lor}\\
x\ot\gl & \mapsto\gl\left(x\right)
\end{align*}
is an isomorphism as étale sheaves over $S$. Let $t\in\mathbb{N}$
be positive integer. Locally in the étale site of $S$ we have an
$\O_{L}$-module isomorphism
\[
\eta:\frac{\O_{L}}{t\O_{L}}\to\frac{\mathcal{M}_{A}}{t\mathcal{M}_{A}}
\]
and let $\gl_{t}\in\mathcal{M}_{A}$ be any lift of $\eta\left(1\right)$.
Looking at the commutative diagram
\[
\xymatrix{A\left[t\right]\ar[d]\ar[r] & A\ot_{\O_{L}}\frac{\mathcal{M}_{A}}{t\mathcal{M}_{A}}\\
A\left[t\right]\ot_{\O_{L}}\frac{\O_{L}}{t\O_{L}}\ar[r] & A\left[t\right]\ot_{\O_{L}}\frac{\mathcal{M}_{A}}{t\mathcal{M}_{A}}\ar[u]
}
\]
we conclude that the map
\begin{align*}
A\left[t\right] & \to A\ot_{\O_{L}}\frac{\mathcal{M}_{A}}{t\mathcal{M}_{A}}\\
x & \mapsto x\ot\overline{\gl_{t}}
\end{align*}
is an isomorphism. Define $\xi_{t}$ using the diagram of isomorphisms
\[
\xymatrix{A\left[t\right]\ar[r]^{\xi_{t}}\ar[d] & A^{\lor}\left[t\right]\\
A\ot_{\O_{L}}\frac{\mathcal{M}_{A}}{t\mathcal{M}_{A}} & \left(A\ot_{\O_{L}}\mathcal{M}_{A}\right)\left[t\right]\ar[l]\ar[u]
}
\]
getting that 
\[
\gl_{t|A\left[t\right]}:A\left[t\right]\simeq A^{\lor}\left[t\right]
\]
is an isomorphism. It follows from Sylow's first Theorem that $\ker\left(\gl_{t}\right)$
has order coprime with $t$ and hence we proved
\begin{prop}
Let $A_{/S}$ be an abelian scheme with ${\rm RM}$ that satisfies
condition (DP). Then for every integer $t\text{\ensuremath{\neq}0}$
it admits a polarisation of degree prime to $t$.
\end{prop}

This fact is interesting because it allows to assume our $p$-divisible
groups are principally polarised. Indeed let $A\to S=\spec\left(R\right)$
be a $g$-dimensional abelian scheme over a $p$-adically complete
ring $R$ and let $f:A\to A^{^{\lor}}$ be an isogeny whose degree,
say $k$, is prime to $p$ with kernel $K$. Let $U\to S$ be an\emph{
}$fppf$ open over which the sequence
\[
0\to K\left(U\right)\to A\left(U\right)\overset{f_{U}}{\longrightarrow}A^{\lor}\left(U\right)\to0
\]
is an exact sequence in abelian groups (or modules over a fixed commutative
ring $\O$). For every integer $n\ge1$ note that the $p^{n}$-torsion
functor $\left(\b\right)\left[p^{n}\right]=\Hom_{\O}\left(\frac{\O}{p^{n}},\b\right)$,
then we have the associated long exact sequence
\[
0\to K\left(U\right)\left[p^{n}\right]\to A\left(U\right)\left[p^{n}\right]\overset{f_{U}}{\longrightarrow}A^{\lor}\left(U\right)\left[p^{n}\right]\to\ext_{\O}^{1}\left(\frac{\O}{p^{n}},K\left(U\right)\right),
\]
but
\[
K\left(U\right)\left[p^{n}\right]=\ext_{\O}^{1}\left(\frac{\O}{p^{n}},K\left(U\right)\right)=0
\]
since $p^{n}$ acts both as 0 and as an invertible element (being
$k$ prime to $p$). In conclusion 
\[
f_{p}:A\left[p^{\infty}\right]\to A^{\lor}\left[p^{\infty}\right]
\]
is an isomorphism if $p$-divisible $\O$-modules.
\begin{note}
\label{note:autodualit=0000E0}In our case this remark applies as
follows: let $A\to S$ be an object parametrised by ${\rm M}\left(\mu_{N},\mathfrak{c}\right)$,
in particular $\mathfrak{c}\ot_{\Z}\Z_{p}$ has rank 1 as an $\O_{L}\ot_{\Z}\Z_{p}$-module
and every element $x\in\mathfrak{c}$ such that $x\ot_{\Z}1$ is a
generator gives rise to a polarisation $A\to A^{\lor}$ whose degree
is necessarily prime to $p$. Suppose that $A$ admits a canonical
subgroup $H_{n}=H_{n}\left(A\right)$ of level $n$, then $A/H_{n}\left(A\right)$
gives a point in ${\rm M}\left(\mu_{N},p^{n}\mathfrak{c}\right)$
(cfr. \cite[pag. 88]{AnGo}) therefore we can induce a prime-to-$p$
polarisation by considering the element $p^{n}x$. Hence we fix such
an element $x\in\mathfrak{c}$ once an for all. In view of the previous
discussion we conclude that all the $p$-divisible $\O_{L}$-modules
$\frac{A}{H_{n}\left(A\right)}\left[p^{\infty}\right]$ come with
compatible isomorphisms (principal $\O_{L}$-polarisations)
\[
\frac{A}{H_{n}\left(A\right)}\left[p^{\infty}\right]\to\left(\frac{A}{H_{n}\left(A\right)}\right)^{\lor}\left[p^{\infty}\right].
\]
\end{note}

\section{The sheaves for $p$ unramified}

We introduced the formal scheme ${\rm Bl}_{\m}\frak{W}^{0}\to\spf\left(\Z_{p}\right)$
as a formal model for the adic weight space. For $\ga\in\m\backslash\m^{2}$
and an interval $I\su\Q_{\ge0}$, then the rings $B_{\ga,I}^{0}$
satisfy the conditions of \ref{subsec:Canonical-subgroups}, hence
we end up with formal schemes
\[
\IG_{n,r,I}^{\left(i\right)}\to\X_{r,I}\to\spf\left(B_{\ga,I}^{0}\right).
\]

Fix an integer $r\ge1$, an interval $I=\left[p^{k},p^{s}\right]$
for two integers $s\ge k\ge0$ and $n\le r+k$. We let $\X$ denote
$\X_{r,I}$, $\IG_{n}^{\left(i\right)}$ denote $\IG_{n,r,I}^{\left(i\right)}$
etc... We let $\A=\A_{r,I}$ be the universal semi-abelian scheme
over $\X$ and $\go_{\A}$ the corresponding sheaf of invariant differentials.
\begin{rem}
Under these assumptions the level-$n$ canonical subgroup $H_{n}\su\A\left[p^{n}\right]$
is defined and it comes with an $\O_{L}$-linear structure compatible
with that of $\A$. Since $H_{n}$ is $p^{n}$-torsion, the map $\go_{\A}\to\go_{H_{n}}$
factors through $\go_{\A}/p^{n}\go_{\A}$ and the kernel of this last
map is anihilated by $\Hdg^{\frac{p^{n}-1}{p-1}}$, in particular
we have a sequence of epimorphisms 
\[
\go_{\A}\to\go_{H_{n}}\to\frac{\go_{\A}}{\gb_{n}\uo_{\A}}
\]
as \emph{fppf} sheaves in $\O_{L}\ot_{\Z}\O_{\X}$-modules over $\X$.
\end{rem}

\begin{prop}
\label{prop:propriet=0000E0DLOGedOmega}Consider the universal morphism
\[
\dlog_{H_{n}}:H_{n}^{D}\to\go_{H_{n}}
\]
of abelian fppf sheaves on $\IG_{n}$ and let $P_{n}$ be the universal
generator of $H_{n}^{D}$. Denote $\Om_{\A}$ the sub-$\O_{\IG_{n}}\ot_{\Z}\O_{L}$-module
of $\go_{\A}$ generated by the lifts of
\[
s=\dlog_{H_{n}}\left(P_{n}\right)\in\frac{\go_{\A}}{p^{n}\Hdg^{-\frac{p^{n}-1}{p-1}}\go_{\A}},
\]
then $\Om_{\A}$ is a locally free $\O_{\IG_{n}}\ot_{\Z}\O_{L}$-module
of rank 1.
\end{prop}

\begin{proof}
This is \cite[Proposition 4.1, pag. 15]{AIP1}.
\end{proof}

\subsection{The Gauss-Manin connection}
\begin{note}
\label{note:connessioneintegrale}Let $X\to\spec\left(\O_{L^{{\rm Gal}}}\left[d_{L}^{-1}\right]\right)=\spec\left(A\right)$
be a scheme of finite type and let $M$ be a coherent locally free
$\O_{X}$-module with a connection $\nabla:M\to M\ot_{\O_{X}}\Omega_{X/A}^{1}$.
Consider the base-change
\[
\nabla_{\Q}:M\ot_{\Z}\Q\to\left(M\ot_{\Z}\Q\right)\ot_{\Q}\Omega_{X_{\Q}/L^{{\rm Gal}}}^{1},
\]
where we used the fact that $\Q\ot_{\Z}\Q=\Q$, and suppose that,
for every $\gamma:L\to L^{{\rm Gal}}$ we have 
\[
\nabla_{\Q}\left(\left(M\ot_{\Z}\Q\right)\left(\gamma\right)\right)\su\left(M\ot_{\Z}\Q\right)\left(\gamma\right)\ot_{\O_{X_{\Q}}}\Omega_{X_{\Q}/L^{{\rm Gal}}}^{1}.
\]
Note that we have
\begin{align*}
\left(M\ot_{\Z}\Q\right)\left(\gamma\right)\ot_{\O_{X_{\Q}}}\Omega_{X_{\Q}/L^{{\rm Gal}}}^{1} & =\left(M\left(\gamma\right)\ot_{\O_{X}}\Omega_{X/A}^{1}\right)\ot_{\Z}\Q,
\end{align*}
therefore 
\[
\nabla\left(M\left(\gamma\right)\right)\su\left(M\ot_{\O_{X}}\Omega_{X/A}^{1}\right)\cap\left[\left(M\left(\gamma\right)\ot_{\O_{X}}\Omega_{X/A}^{1}\right)\ot_{\Z}\Q\right]\su M\left(\gamma\right)\ot_{\O_{X}}\Omega_{X/A}^{1}.
\]
\end{note}

\begin{lem}
\label{lem:GaussManinsuL}Let $\mathbf{M}\to\spec\left(\O_{L^{{\rm Gal}}}\left[d_{L}^{-1}\right]\right)$
be the base change of a smooth toroidal compactification of $M\left(\mu_{N},\c\right)$
\[
\nabla:\h=\h\left(\mathbf{A}/\mathbf{M}\right)\to\h\ot_{\O_{\mathbf{M}}}\Omega_{\mathbf{M}/\O_{L^{{\rm Gal}}}\left[d_{L}^{-1}\right]}^{1}
\]
be the Gauss-Manin connection. Then for every $\gamma:L\to L^{{\rm Gal}}$
we have
\[
\nabla\left(\h\left(\gamma\right)\right)\su\h\left(\gamma\right)\ot_{\O_{\mathbf{M}}}\Omega_{\mathbf{M}/\O_{L^{{\rm Gal}}}\left[d_{L}^{-1}\right]}^{1}.
\]
\end{lem}

\begin{proof}
We follow \cite[Lemma 2.1.14, pag 229]{Kat}. Let $\Sigma={\rm Hom}_{\Q}\left(L,L^{{\rm Gal}}\right)$.
Embed $L^{{\rm Gal}}\su\C$ and consider the base change $\mathbf{M}_{\mathbb{C}}\to\spec\left(\C\right)$,
whose associated analytic variety has uniformisation by $\frak{h}\left(L\right)=\frak{h}^{\Sigma}$,
where $\frak{h}$ is the usual complex half plane (\cite[Section 1.4, pag. 213]{Kat}).
Over $\frak{h}\left(L\right)$ we have a horizontal isomorphism between
$\h$ and a constant sheaf 
\[
V\ot_{\C}\O_{\frak{h}\left(L\right)}\simeq\h
\]
with the trivial connection ${\rm id}_{V}\ot_{\C}d$. For every $\gamma\in\Sigma$
let $\left\{ X_{\gamma},Y_{\gamma}\right\} $ be a basis of horizontal
sections of $\h\left(\gamma\right)$. Then Kodaira-Spencer isomorphism
reads
\begin{align*}
\Omega_{\frak{h}\left(L\right)/\C}^{1} & \to\uo^{\ot2}=\bigoplus_{\s}\left(X_{\s}-\tau_{\s}Y_{\s}\right)^{2}\cdot\O_{\frak{h}\left(L\right)}\\
2\pi i\,d\tau_{\s} & \mapsto\left(X_{\s}-\tau_{\s}Y_{\s}\right)^{2}
\end{align*}
Let $\tilde{D}_{\s}:\h\to\h\ot\omega\left(\s^{2}\right)$ as in \cite[Section 2.1]{Kat},
then
\[
\tilde{D}_{\s}=\frac{1}{2\pi i}\left(X_{\s}-\tau_{\s}Y_{\s}\right)^{2}\frac{\partial}{\partial\tau_{\s}},
\]
hence if $\xi=fX_{\gamma}+gY_{\gamma}\in\h\left(\gamma\right)$ we
have
\[
\tilde{D}_{\s}\left(\xi\right)=\frac{1}{2\pi i}\left(\frac{\partial f}{\partial\tau_{\s}}X_{\gamma}+\frac{\partial g}{\partial\tau_{\s}}Y_{\gamma}\right)\otimes_{\C}\left(X_{\s}-\tau_{\s}Y_{\s}\right)^{2}\in\h\left(\gamma\right)\ot\omega\left(\s^{2}\right).
\]
Since have $\nabla=\sum_{\s}\tilde{D}_{\s}$ we conclude that the
statements holds on $\mathbf{M}_{L}$ and in view of Note \ref{note:connessioneintegrale}
that it holds over $\mathbf{M}$.
\end{proof}
Back to our setting we conclude that
\begin{prop}
\label{prop:gaussmaninrispettaO_F}Let $\nabla:\h\to\h\widehat{\ot}_{\O_{\X}}\Omega_{\X/S}^{1}$
be the Gauss-Manin connection. Then
\begin{enumerate}
\item $\nabla$ restricts to an integrable connection
\[
\nabla^{\left(i\right)}:H_{{\rm dR}}^{1,\left(i\right)}\to H_{{\rm dR}}^{1,\left(i\right)}\widehat{\otimes}_{\O_{\X}}\Omega_{\X/S}^{1}
\]
for every $i$;
\item the connection $\nabla^{\left(i\right)}$ respects the respects the
$\O_{F_{i}}$-structure.
\end{enumerate}
\end{prop}

\subsection{Descending to $\protect\IG_{n}$}

Define the formal group $\frak{T}=1+\gb_{n}\cdot\Res_{\O_{L}/\Z}\left(\G_{a}\right)$,
hence over $\IG_{n}^{\left(i\right)}$ the universal character gives
a morphism $\kappa^{\left(i\right)}:\frak{T}^{\left(i\right)}=1+\gb_{n}\Res_{\O_{F_{i}}|\Z_{p}}\left(\G_{a}\right)\to\G_{m|\IG_{n}^{\left(i\right)}}$
for every $i$. 
\begin{prop}
\label{prop:discesaIGngenerale}Let $\left(\E,s\right)$ be an ${\rm MS}_{\O_{Fi}}$-datum
of rank $m$ on $\IG_{n}^{\left(i\right)}$ and denote with $\pi:\V_{0}\left(\E,s\right)\to\IG_{n}^{\left(i\right)}$
the formal $\O_{F_{i}}$-module bundle with marked sections associated
to it, then
\begin{enumerate}
\item the map $\pi$ has a natural action of $\frak{T}^{\left(i\right)}$;
\item let $\spf\left(R\right)\su\IG_{n}^{\left(i\right)}$ be an open subset
over which $\E$ and $\gb_{n}$ are free, then
\[
\pi{}_{\ast}\O_{\V_{0,\s}\left(\E,s\right)}\left[\kappa{}^{\left(i\right)}\right]_{|\spf\left(R\right)}=\prod_{u\in S}R\left\langle V_{2,\s,u},\dots,V_{m,\s,u}\right\rangle \cdot\kappa{}^{\left(i\right)}\left(\s\left(u\right)+\gb_{n}Z_{\s,u}\right)
\]
where $V_{k,\s,u}=\left(\s\left(u\right)+\gb_{n}Z_{\s,u}\right)^{-1}X_{k,\s,u}$.
\end{enumerate}
\end{prop}

\begin{proof}
For fixed $\s$ and $u$ the proof of \cite[Lemma 3.13]{AI} shows
that 
\[
\pi{}_{\ast}\O_{\V_{0,\s}^{u}\left(\E,s\right)}\left[\kappa{}^{\left(i\right)}\right]_{|\spf\left(R\right)}=R\left\langle V_{2,\s,u},\dots,V_{m,\s,u}\right\rangle \cdot\kappa{}^{\left(i\right)}\left(\s\left(u\right)+\gb_{n}Z_{\s,u}\right).
\]
We conclude since 
\[
\pi{}_{\ast}\O_{\V_{0,\s}\left(\E,s\right)}\left[\kappa{}^{\left(i\right)}\right]=\prod_{u\in S}\pi{}_{\ast}\O_{\V_{0,\s}^{u}\left(\E,s\right)}\left[\kappa{}^{\left(i\right)}\right].
\]
\end{proof}
Note that, since $\Om_{\A}$ is an invertible $\O_{L}\ot_{\Z}\O_{\IG_{n}}$-module,
$\Omega_{\A}^{\left(i\right)}$ is an invertible $\O_{F_{i}}\ot_{\Z_{p}}\O_{\IG_{n}^{\left(i\right)}}$-module.
\begin{defn}
Define $\varpi^{\kappa^{\left(i\right)}}=\varpi_{n,r,I}^{\kappa^{\left(i\right)}}=\left(\pi_{\ast}\O_{\V_{0}\left(\Omega_{\A}^{\left(i\right)},{\bf s}\right)}\right)\left[\kappa^{\left(i\right)}\right]$,
that is the subsheaf of $\pi_{\ast}\O_{\V_{0}\left(\Omega_{\A}^{\left(i\right)},{\bf s}\right)}$
consisting of sections transforming according to the character $\kappa^{\left(i\right)}$
under the action of $\frak{T}^{\left(i\right)}$.
\end{defn}

\begin{rem}
It follows from Proposition \ref{prop:discesaIGngenerale} that $\varpi^{\kappa^{\left(i\right)}}$
is a locally free $\O_{\IG_{n}^{\left(i\right)}}$-module of rank
$\#S$.

What we did here for the character $\kappa$ holds verbatim when $\kappa$
is replaced by any locally analytic character $\chi:\frak{T}\to\G_{m}$.
\end{rem}

\begin{cor}
\label{cor:identificazionepesointero}With setting and notations as
in Proposition \ref{prop:discesaIGngenerale} , let $\E$ be an invertible
$\O_{F_{i}}\ot_{\Z_{p}}\O_{\IG_{n}^{\left(i\right)}}$-module. Then
for every integer $k$ we have
\[
\pi{}_{\ast}\O_{\V_{0}\left(\E,s\right)}\left[k\right]\simeq\prod_{u\in S}\E^{\ot k}.
\]
\end{cor}

\begin{note}
It is shown in \cite[Proposition A.3, pag. 43]{AIP2} that there exists
an invertible ideal $\ud\su\O_{\IG_{1}}$ with the property that $\ud^{p-1}=\Hdg$
and $\ud\cdot\uo_{\A}=\Omega_{\A}$. 
\end{note}

Consider now the Hodge filtration on $\H_{\dR}^{1}\left(\A/\IG_{n}\right)$,
this is given by the exat sequence
\[
\HH_{\A}^{\b}\quad=\quad0\to\uo_{\A}\to\H_{\dR}^{1}\left(\A/\IG_{n}\right)\to\uo_{\A^{D}}^{\lor}\to0
\]
and $\H_{\dR}^{1}\left(\A/\IG_{n}\right)$ comes with the Gauss-Manin
connection. We let $H_{\A}^{\#}$ be the pushout of the diagram 
\begin{equation}
\xymatrix{\ud^{p}\cdot\uo_{\A}\ar[r]\ar[d] & \ud^{p}\cdot\H_{\dR}^{1}\left(\A/\IG_{n}\right)\ar@{.>}[d]\\
\Omega_{\A}\ar@{.>}[r] & H_{\A}^{\#}
}
\label{eq:pushout-1}
\end{equation}
The following Lemma is a consequence of the formal properties of pushouts
\begin{lem}
\label{lem:descrizioneesplicitapushout}$\,$
\begin{enumerate}
\item We have $\ud^{p}\cdot\H_{\dR}^{1}\left(\A/\IG_{n}\right)\cap\Omega_{\A}=\ud^{p}\cdot\uo_{\A}$as
subsheaves of $\H_{\dR}^{1}\left(\A/\IG_{n}\right)$, so in particular
\[
H_{\A}^{\#}=\underline{\delta}^{p}\cdot\H_{\dR}^{1}\left(\A/\IG_{n}\right)+\Omega_{\A}\su\H_{\dR}^{1}\left(\A/\IG_{n}\right);
\]
\item We have an exact sequence of $\O_{L}\ot_{\Z}\O_{\IG_{n}}$-modules
\[
\left({\cal H}_{\A}^{\#}\right)^{\b}\quad=\quad0\to\Omega_{\A}\to H_{\A}^{\#}\to\ud^{p}\cdot\uo_{\A^{D}}^{\lor}\to0,
\]
in particular $H_{\A}^{\#}$ is a locally free $\O_{L}\ot_{\Z}\O_{\IG_{n}}$-module
of rank 2 and there exists a natural monomorphism $j^{\b}:\left({\cal H}_{\A}^{\#}\right)^{\b}\to\HH_{\A}^{\b}$.
\end{enumerate}
\end{lem}

\begin{prop}[{\cite[Proposition 6.2, pag. 64]{AI}}]
\label{prop:propriet=0000E0fascioH=000023-1}Let $P_{n}\in H_{n}^{\left(i\right),D}\left(\IG_{n}^{\left(i\right)}\right)$
be an $\O_{F_{i}}/\p_{i}^{n}$- basis and let $s_{i}=\dlog\left(P_{n}\right)$.
Then the exact sequence $\left({\cal H}_{\A}^{\#,\left(i\right)}\right)^{\b}$
realises $\left(\Omega_{\A}^{\left(i\right)},s_{i}\right)$ as an
${\rm MS}_{\O_{F_{i}}}$-subdatum of $\left(H_{\A}^{\#,\left(i\right)},s_{i}\right)$
with respect to the ideal $\gb_{n}$. Moreover the Gauss-Manin connection
$\nabla^{\left(i\right)}$ on $\H_{\dR}^{1}\left(\A/\IG_{n}\right)^{\left(i\right)}$
induces a connection
\[
\nabla^{\#,\left(i\right)}:H_{\A}^{\#,\left(i\right)}\to H_{\A}^{\#,\left(i\right)}\widehat{\ot}_{\O_{\IG_{n}^{\left(i\right)}}}\ud^{\left(i\right),-1}\cdot\Omega_{\IG_{n}^{\left(i\right)}/\X}^{1}.
\]
\end{prop}

\begin{proof}
The first statement is a direct consequence of the fact the sequence
$\left({\cal H}_{\A}^{\#}\right)^{\b}$ is locally split, since $\ud^{p}\cdot\uo_{\A^{D}}^{\lor}$
is locally free. The second statement follows from the proof of \cite[Proposition 6.3, pag. 64]{AI}.
\end{proof}
\begin{defn}
In view of Proposition \ref{prop:propriet=0000E0fascioH=000023-1}
we can consider $\pi^{\left(i\right)}:\V^{\left(i\right)}:=\V_{0}\left(H_{\A}^{\#,\left(i\right)},s_{i}\right)\to\IG_{n}^{\left(i\right)}$
. We define $\frak{f}^{\left(i\right)}$ to be the composition
\[
\frak{f^{\left(i\right)}}:\V^{\left(i\right)}\overset{\pi^{\left(i\right)}}{\longrightarrow}\IG_{n}^{\left(i\right)}\overset{\frak{h}^{\left(i\right),n}}{\longrightarrow}\X.
\]
\end{defn}

\begin{rem}
It follows from Proposition \ref{prop:discesaIGngenerale} that the
$\O_{\IG_{n}^{\left(i\right)}}$-module $\tilde{\mathbb{W}}_{\kappa^{\left(i\right)}}:=\pi_{\ast}^{\left(i\right)}\O_{\V^{\left(i\right)}}\left[\kappa^{\left(i\right)}\right]$
has locally the form
\[
\tilde{\mathbb{W}}_{\kappa^{\left(i\right)}|\spf\left(R\right)}=\prod_{u\in S}R\left\langle V_{\s,u}\right\rangle \cdot\kappa^{\left(i\right)}\left(\s\left(u\right)+\gb_{n}Z_{\s,u}\right).
\]
Moreover we can endow the $\O_{\IG_{n}^{\left(i\right)}}$-module
$\tilde{\mathbb{W}}_{\kappa^{\left(i\right)}}$ with a natural increasing
filtration
\[
F^{\b}\tilde{\mathbb{W}}_{\kappa^{\left(i\right)}}=\pi_{\ast}^{\left(i\right)}F^{\b}\O_{\V^{\left(i\right)}}\left[\kappa^{\left(i\right)}\right]
\]
induced by $\left(\Omega_{\A}^{\left(i\right)},s_{i}\right)$ with
the property that
\[
\gr^{h}F^{\b}\tilde{\mathbb{W}}_{\kappa^{\left(i\right)}}=\pi_{\ast}\O_{\V_{0}\left(\Omega_{\A}^{\left(i\right)},s_{i}\right)}\left[\kappa^{\left(i\right)}\right]\ot\gr^{h}F^{\b}\pi_{\ast}^{\left(i\right)}\O_{\V_{\O_{F_{i}}}\left(\left(\ud^{p}\cdot\uo_{\A^{D}}^{\lor}\right)^{\left(i\right)}\right)}.
\]
\end{rem}

\begin{prop}
\label{prop:strutturaWtilde}The following hold:
\begin{enumerate}
\item $F^{h}\tilde{\mathbb{W}}_{\kappa^{\left(i\right)}}$ is a locally
free coherent $\O_{\IG_{n}^{\left(i\right)}}$-module for every $h\ge0$;
\item $\tilde{\mathbb{W}}_{\kappa^{\left(i\right)}}$ is isomorphism to
the completed limit $\widehat{\varinjlim}F^{h}\tilde{\mathbb{W}}_{\kappa^{\left(i\right)}}$;
\item $F^{0}\tilde{\mathbb{W}}_{\kappa^{\left(i\right)}}\simeq\varpi^{\kappa^{\left(i\right)}}$
and ${\rm Gr}^{h}F^{\b}\tilde{\mathbb{W}}_{\kappa^{\left(i\right)}}\simeq\varpi^{\kappa^{\left(i\right)}}\ot_{\IG_{n}^{\left(i\right)}}\left(\uo_{\A}^{\left(i\right)}\right)^{-h}\ot_{\IG_{n}^{\left(i\right)}}\left(\uo_{\A^{D}}^{\left(i\right)}\right)^{-h}$.
In particular, we have locally
\[
F^{h}\tilde{\mathbb{W}}_{\kappa^{\left(i\right)}}\left(\spf\left(R\right)\right)=\varpi^{\kappa^{\left(i\right)}}\ot_{R}\left(\bigoplus_{0\le k\le h}\left(\uo_{\A}^{\left(i\right)}\right)^{-k}\ot_{\IG_{n}^{\left(i\right)}}\left(\uo_{\A^{D}}^{\left(i\right)}\right)^{-k}\right).
\]
\end{enumerate}
\end{prop}

\begin{proof}
From the explicit description of the filtration we have locally on
$\spf\left(R\right)\su\IG_{n}^{\left(i\right)}$
\[
F^{h}\tilde{\mathbb{W}}_{\kappa^{\left(i\right)}|\spf\left(R^ {}\right)}=\prod_{u\in S}R\left[V_{\s,u}\right]_{\le h}\cdot\kappa^{\left(i\right)}\left(\s\left(u\right)+\gb_{n}Z_{\s,u}\right)
\]
from which points 1. and 2., while the first part of point 3. comes
from Proposition \ref{prop:discesaIGngenerale}. To conclude, in view
again of Proposition \ref{prop:discesaIGngenerale}, we can write
locally
\begin{align*}
{\rm Gr}^{h}F^{\b}\tilde{\mathbb{W}}_{\kappa^{\left(i\right)}}\left(\spf\left(R\right)\right) & =\prod_{u\in S}\bigotimes_{\s\in\sS}R\cdot\kappa^{\left(i\right)}\left(\s\left(u\right)+\gb_{n}Z_{\s,u}\right)V_{\s,u}^{h}\\
 & =\prod_{u\in S}R\cdot\left(\prod_{\s\in\sS}\kappa^{\left(i\right)}\left(\s\left(u\right)+\gb_{n}Z_{\s,u}\right)\right)\left(\prod_{\s\in\sS}\left(\s\left(u\right)+\gb_{n}Z_{\s,u}\right)^{-h}\right)\left(\prod_{\s\in\sS}X_{2,\s,u}^{h}\right)\\
 & =\left(\varpi^{\kappa^{\left(i\right)}}\ot_{\IG_{n}^{\left(i\right)}}\left(\uo_{\A}^{\left(i\right)}\right)^{-h}\ot_{\IG_{n}^{\left(i\right)}}\left(\uo_{\A^{D}}^{\left(i\right)}\right)^{-h}\right)\left(\spf\left(R\right)\right).
\end{align*}
\end{proof}

\subsection{Descending to $\protect\X$}

Define the formal group $\frak{T}^{{\rm ext}}=\T\left(\Z_{p}\right)\cdot\frak{T}$.
Over $\X$ it decomposes as
\[
\prod_{i=1}^{d}\frak{T}^{{\rm ext},\left(i\right)}=\prod_{i=1}^{d}\O_{F_{i}}\inv\cdot\left(1+\gb_{n}{\rm Res}_{\O_{F_{i}}|\Z_{p}}\G_{a}\right)
\]
The group $\frak{T}^{{\rm ext},\left(i\right)}$ comes with a natural
action on $\frak{f}^{\left(i\right)}:\V_{0}\left(\Omega_{\A}^{\left(i\right)},s_{i}\right)\to\X$
that we now describe. Let $u:S\to\X$ be a morphism of formal schemes,
then a point of $\V_{0}\left(\Omega_{\A}^{\left(i\right)},s_{i}\right)\left(u\right)$
is a pair $\left(\rho^{\left(i\right)},v\right)$ where $\rho^{\left(i\right)}:S\to\IG_{n}^{\left(i\right)}$
lifts $u$ and $v\in\V_{0}\left(\Omega_{\A}^{\left(i\right)},s_{i}\right)\left(\rho^{\left(i\right)}\right)$.
In the same fashion a point of $\V_{0}\left(H_{\A}^{\#,\left(i\right)},s_{i}\right)\left(u\right)$
is a pair $\left(\rho^{\left(i\right)},w\right)$ where $\rho^{\left(i\right)}:S\to\IG_{n}^{\left(i\right)}$
lifts $u$ and $w\in\V_{0}\left(H_{\A}^{\#,\left(i\right)},s_{i}\right)\left(\rho^{\left(i\right)}\right)$.
Let $\overline{\gl}\in\left(\O_{F_{i}}/\p_{i}^{n}\right)\inv$ be
seen as an element of the Galois group of the adic generic fibre ${\cal IG}_{n}^{\left(i\right)}\to{\cal X}$
and, by functoriality, as an $\X$-automorphism of $\IG_{n}^{\left(i\right)}$.
Denote 
\[
\overline{\gl}^{\ast}:\Omega_{\A}^{\left(i\right)}\to\Omega_{\A}^{\left(i\right)},\quad H_{\A}^{\#,\left(i\right)}\to H_{\A}^{\#,\left(i\right)}
\]
 the isomorphisms it induces, which are characterised $\mod\gb_{n}$
by $\overline{\gl}^{\ast}\left(s_{i}\right)=\overline{\gl}s_{i}$.
For a point $\left(\rho^{\left(i\right)},v\right)\in\V_{0}\left(\Omega_{\A}^{\left(i\right)},s_{i}\right)\left(u\right)$
we define 
\begin{equation}
\gl\ast\left(\rho^{\left(i\right)},v\right)=\left(\overline{\gl}\circ\rho^{\left(i\right)},\gl^{-1}v\right)\label{eq:azionebrutta}
\end{equation}
and for a point $\left(\rho^{\left(i\right)},w\right)\in\V_{0}\left(H_{\A}^{\#,\left(i\right)},s_{i}\right)\left(u\right)$
we define 
\begin{equation}
\gl\ast\left(\rho^{\left(i\right)},w\right)=\left(\overline{\gl}\circ\rho^{\left(i\right)},\gl^{-1}w\right).\label{eq:azionebrutta-1}
\end{equation}
As in the proof of \cite[Lemme 5.1, pag. 21]{AIP2}, formula (\ref{eq:azionebrutta})
defines an action of $\T\left(\Z_{p}\right)$ on $\frak{f}^{\left(i\right)}:\V_{0}\left(\Omega_{\A}^{\left(i\right)},s_{i}\right)\to\X$
which is compatible with that of $\frak{T}$ on $\pi^{\left(i\right)}$,
thus giving an action of $\frak{T}^{{\rm ext},\left(i\right)}$ on
$\frak{f}^{\left(i\right)}$. 
\begin{defn}
Let $\frak{w}^{\kappa^{\left(i\right)}}=\left(\frak{f}_{\ast}^{\left(i\right)}\O_{\V_{0}\left(\Omega_{\A}^{\left(i\right)},s_{i}\right)}\right)\left[\kappa^{\left(i\right)}\right]$
that is, $\frak{w}^{\kappa^{\left(i\right)}}$ is the $\O_{\X}$-submodule
of $\frak{f}_{\ast}^{\left(i\right)}\O_{\V_{0}\left(\Omega_{\A}^{\left(i\right)},s_{i}\right)}$
given by sections that tranforms via $\kappa^{\left(i\right)}$ under
the action of $\frak{T}^{{\rm ext},\left(i\right)}$. 
\end{defn}

\begin{rem}
\label{rem:Isomorfismotorsore}In \cite[Section 4.1, pag. 15]{AIP1}
the $\frak{T}$-torsor $f:\frak{F}_{n}\to\IG_{n}$ is considered,
defined by
\[
\frak{F}_{n}\left(R\right)=\left\{ \omega\in\Omega_{\A}\left(R\right)\,\middle|\,\overline{\omega}=s\right\} .
\]
Let $\spf\left(R\right)\su\IG_{n}^{\left(i\right)}$ be connected,
then for every $\omega_{\s}^{\left(i\right)}\in\frak{F}_{n,\s}^{\left(i\right)}\left(R\right)$
we can consider the map $\Omega_{\A}^{\left(i\right)}\left(\s\right)\to R$
defined by $f_{\s}\left(\omega_{\s}^{\left(i\right)}\right)=\s\left(u\right)$.
This gives a map 
\[
\phi_{n,\s}^{\left(i\right)}:\frak{F}_{n,\s}^{\left(i\right)}\to\V_{0,\s}^{u}\left(\Omega_{\A}^{\left(i\right)},s_{i}\right)
\]
over $\IG_{n}^{\left(i\right)}$. On the other hand, given a section
$f_{\s}\in\V_{0,\s}^{u}\left(\Omega_{\A}^{\left(i\right)},s_{i}\right)\left(R\right)$
then $\s\left(u\right)f_{\s}^{\lor}\in\Omega_{\A}^{\left(i\right)}\left(\s\right)$
lifts $s$ (since $\overline{f}_{\s}\left(\overline{\s\left(u\right)}\cdot\overline{f_{\s}^{\lor}}\right)=\overline{f}_{\s}\left(s\right)$)
and this gives a map 
\[
\V_{0,\s}^{u}\left(\Omega_{\A}^{\left(i\right)},s_{i}\right)\to\frak{F}_{n,\s}^{\left(i\right)}
\]
again over $\IG_{n}^{\left(i\right)}$ and it is clear from the functorial
description that this is indeed the inverse of $\phi_{n,\s}^{\left(i\right)}$.
If $h\in\frak{T}^{\left(i\right)}\left(R\right)$, then $\phi_{n,\s}^{\left(i\right)}\left(h\omega_{\s}^{\left(i\right)}\right)$
is defined by $h\omega_{\s}^{\left(i\right)}\mapsto\s\left(u\right)$,
hence we conclude that
\[
\phi_{n,\s}^{\left(i\right)}\left(h\omega_{\s}^{\left(i\right)}\right)=h^{-1}\phi_{n,\s}^{\left(i\right)}\left(\omega_{\s}^{\left(i\right)}\right).
\]
We proved that for every $u\in S$ there is a natural isomorphism
\[
\phi_{n}:\frak{F}_{n}\to\V_{0}^{u}\left(\Omega_{\A},s\right)
\]
over $\IG_{n}$ that inverts the action of $\frak{T}$. The importance
of this isomorphism lies in the fact that from \cite[Section 4.2, pag. 16]{AIP1}
the sheaf of overconvergent forms is
\[
\frak{w}_{n,r,I}^{{\rm over}}:=\left(\frak{h}^{n}\circ f\right)_{\ast}\O_{\frak{F}_{n}}\left[\kappa^{-1}\right]
\]
over $\X_{r,I}$. Hence we obtain that
\[
\frak{w}_{n,r,I}^{\kappa}=\prod_{u\in S}\frak{w}_{n,r,I}^{{\rm over}}.
\]
\end{rem}

\begin{note}
\label{note:provache=00005Cfrakwinvertibile}Let $\spf\left(A\right)\su\X$
be an open affine over which $\Hdg$ is free, say generated by $\tilde{\Ha}$,
and let $\spf\left(R_{n}^{\left(i\right)}\right)\su\IG_{n}^{\left(i\right)}$
be its inverse image. In view of Remark \ref{rem:ramificazioneIgusa},
for $j=0,\dots,n$, there exist elements $c_{j}\in\tilde{\Ha}^{-1}R_{n}^{\left(i\right)}$
such that
\begin{enumerate}
\item $c_{j}\in\tilde{\Ha}^{-\left(p^{j}-p\right)/\left(p-1\right)}R_{n}^{\left(i\right)}$;
\item ${\rm Tr}_{R_{j}/R_{j-1}}\left(c_{j}\right)=c_{j-1}$ (here we see
$R_{j-1}$ as a subring of $R_{j}$) and $c_{0}=1$. 
\end{enumerate}
Pick a generator $\kappa^{\left(i\right)}\left(\s\left(u\right)+\gb_{n}Z_{\s,u}\right)$
of $\varpi_{\s,u}^{\kappa^{\left(i\right)}}$ over $R_{n}^{\left(i\right)}$
(compare with Proposition \ref{prop:discesaIGngenerale}) and a lift
$\tilde{\gamma}\in\O_{F_{i}}\inv$ for every $\gamma\in\left(\O_{F_{i}}/p^{n}\right)\inv$.
Note that the element
\[
b_{\s,u}=\sum_{\gamma\in\left(\O_{F_{i}}/p^{n}\right)\inv}\kappa^{\left(i\right)}\left(\tilde{\gamma}\right)\sigma\ast\left(\kappa^{\left(i\right)}\left(\s\left(u\right)+\gb_{n}Z_{\s,u}\right)c_{n}\right)
\]
a priori lies in $\text{\ensuremath{\tilde{\Ha}}}^{-1}R_{n}^{\left(i\right)}\left\langle Z_{\s,u}\right\rangle $.
In view of Proposition \ref{prop:analiticitadecarattereuniversale}
we can write $\kappa^{\left(i\right)}\left(\s\left(u\right)+\gb_{n}Z_{\s,u}\right)=\k^{\left(i\right)}\left(\s\left(u\right)\right)+qh$
for some $h\in R_{n}^{\left(i\right)}\left\langle Z_{\s,u}\right\rangle $
hence
\[
b_{\s,u}=\kappa^{\left(i\right)}\left(\s\left(u\right)\right)\sum_{\gamma\in\left(\O_{F_{i}}/p^{n}\right)\inv}\kappa^{\left(i\right)}\left(\tilde{\gamma}\right)\gamma\ast c_{n}+q\sum_{\gamma\in\left(\O_{F_{i}}/p^{n}\right)\inv}\kappa^{\left(i\right)}\left(\tilde{\gamma}\right)\gamma\ast\left(hc_{n}\right).
\]
The term $q\sum_{\gamma\in\left(\O_{F_{i}}/p^{n}\right)\inv}\kappa^{\left(i\right)}\left(\tilde{\gamma}\right)\gamma\ast\left(hc_{n}\right)$
lives in $q\tilde{\Ha}^{-\left(p^{n}-p\right)/\left(p-1\right)}R_{n}^{\left(i\right)}\left\langle Z_{\s,u}\right\rangle $,
but $n\le r+k$ implies that $\tilde{\Ha}^{\left(p^{n}-p\right)/\left(p-1\right)}\vert\tilde{\Ha}^{p^{r+k}}$
, but $p\ga^{-p^{k}}\in B_{\ga,I}^{0}$ and $\ga\Hdg^{-p^{r+1}}\su R_{n}^{\left(i\right)}$
and it follows that $\tilde{\Ha}^{p^{r+k+1}}\vert p$. Therefore,
denoting $\left(R_{n}^{\left(i\right)}\right)^{\text{\textdegree\textdegree}}$
the ideal of topologically nilpotent elements in $R_{n}^{\left(i\right)}$,
we have 
\[
q\sum_{\gamma\in\left(\O_{F_{i}}/p^{n}\right)\inv}\kappa^{\left(i\right)}\left(\tilde{\gamma}\right)\gamma\ast\left(hc_{n}\right)\in\left(R_{n}^{\left(i\right)}\right)^{\text{\textdegree\textdegree}}\cdot R_{n}^{\left(i\right)}\left\langle Z\right\rangle .
\]
Consider now $\sum_{\gamma\in\left(\O_{F_{i}}/p^{n}\right)\inv}\kappa^{\left(i\right)}\left(\tilde{\gamma}\right)\gamma\ast c_{n}$:
first note that in view of the choice of the sequence $c_{j}$ and,
since $\kappa^{\left(i\right)}\left(\tilde{\gamma}\right)\in1+\left(R_{n}^{\left(i\right)}\right)^{\text{\textdegree\textdegree}}$
we have 
\begin{align*}
\sum_{\gamma\in1+p^{t-1}\O_{F_{i}}/p^{t}\O_{F_{i}}}\kappa^{\left(i\right)}\left(\tilde{\gamma}\right)\gamma\ast c_{t} & \equiv\sum_{\gamma\in1+p^{t-1}\O_{F_{i}}/p^{t}\O_{F_{i}}}\gamma\ast c_{t}\mod\left(R_{n}^{\left(i\right)}\right)^{\text{\textdegree\textdegree}}\\
 & =c_{t-1}
\end{align*}
therefore 
\begin{align*}
\sum_{\gamma\in\left(\O_{F_{i}}/p^{n}\right)\inv}\kappa^{\left(i\right)}\left(\tilde{\gamma}\right)\gamma\ast c_{n} & \equiv1\mod\left(R_{n}^{\left(i\right)}\right)^{\text{\textdegree\textdegree}}
\end{align*}
and $b_{\s,u}$ is a well-defined element of $\frak{w}_{\s,u}^{\kappa^{\left(i\right)}}=\frak{f}_{\ast}^{\left(i\right)}\O_{\V_{0,\s}^{u}\left(\Omega_{\A}^{\left(i\right)},s_{i}\right)}\left[\kappa^{\left(i\right)}\right]$
over $\spf\left(A\right)$. From topological Nakayama's Lemma \cite[Theorem 8.4, pag. 58]{Mat}
we conclude that $\frak{w}_{u}^{\kappa^{\left(i\right)}}=\bigotimes_{\s}\frak{w}_{\s,u}^{\kappa^{\left(i\right)}}$
is free over $A$ with basis $\bigotimes_{\s}b_{\s,u}$.
\end{note}

\begin{defn}
Set $\mathbb{W}_{\kappa^{\left(i\right)}}=\frak{f}_{\ast}^{\left(i\right)}\O_{\V^{\left(i\right)}}\left[\kappa^{\left(i\right)}\right]$.
\end{defn}

\begin{thm}
\label{thm:filtrazionesuX}With setting and notations as in Note \ref{note:provache=00005Cfrakwinvertibile},
the action of $\frak{T}^{{\rm ext},\left(i\right)}$ on $\frak{f}_{\ast}^{\left(i\right)}\O_{\V^{\left(i\right)}}$
preserves the filtration $F^{\b}\frak{f}_{\ast}^{\left(i\right)}\O_{\V^{\left(i\right)}}$
induced by the ${\rm MS}_{\O_{F_{i}}}$-subdatum $\left(\Omega_{\A}^{\left(i\right)},s_{i}\right)$
of $\left(H_{\A}^{\#,\left(i\right)},s_{i}\right)$. Set
\[
F^{\b}\mathbb{W}_{\kappa^{\left(i\right)}}=\frak{f}_{\ast}^{\left(i\right)}F^{\b}\O_{\V^{\left(i\right)}}\left[\kappa^{\left(i\right)}\right],
\]
then 
\begin{enumerate}
\item $F^{h}\mathbb{W}_{\kappa^{\left(i\right)}}$ is a locally free coherent
$\O_{\X}$-module;
\item $\mathbb{W}_{\kappa^{\left(i\right)}}$ is isomorphic to the completed
limit $\widehat{\varinjlim}F^{h}\mathbb{W}_{\kappa^{\left(i\right)}}$,
in particular $\mathbb{W}_{\kappa^{\left(i\right)}}$ is a flat $\O_{\X}$-module;
\item $F^{0}\mathbb{W}_{\kappa^{\left(i\right)}}\simeq\frak{w}^{\kappa^{\left(i\right)}}$
and ${\rm Gr}^{h}F^{\b}\mathbb{W}_{\kappa^{\left(i\right)}}\simeq\frak{w}^{\kappa^{\left(i\right)}}\ot_{\O_{\X}}\uo_{\A}^{-h}\ot_{\O_{\X}}\uo_{\A^{D}}^{-h}$.
In particular, we have locally
\[
F^{h}\mathbb{W}_{\kappa^{\left(i\right)}}\left(\spf\left(A\right)\right)\simeq\frak{w}^{\kappa^{\left(i\right)}}\ot_{A}{\rm Sym}_{A}^{\le h}\left(\uo_{\A}^{-1}\ot_{A}\uo_{\A^{D}}^{-1}\right).
\]
\end{enumerate}
\end{thm}

\begin{proof}
Note that, in view of Proposition \ref{prop:strutturaWtilde}, points
1. and 2. follow from point 3. The isomorphism $F^{0}\mathbb{W}_{\kappa^{\left(i\right)}}\simeq\frak{w}^{\kappa^{\left(i\right)}}$
comes from the very definition of the filtration
\begin{align*}
F^{0}\mathbb{W}_{\kappa^{\left(i\right)}} & =F^{0}\frak{f}_{\ast}^{\left(i\right)}\O_{\V^{\left(i\right)}}\left[\kappa^{\left(i\right)}\right]\\
 & =\frak{f}_{\ast}^{\left(i\right)}\O_{\V_{0}\left(\Omega_{\A}^{\left(i\right)},s_{i}\right)}\left[\kappa^{\left(i\right)}\right]\\
 & =\frak{w}^{\kappa^{\left(i\right)}}.
\end{align*}
Finally, in view of Proposition \ref{prop:strutturaWtilde}, since
$\omega_{\A}$ is defined on $\X$, we see that 
\begin{align*}
\gr^{h}F^{\b}\mathbb{W}_{\kappa^{\left(i\right)}} & =\frak{f}_{\ast}^{\left(i\right)}\O_{\V_{0}\left(\Omega_{\A}^{\left(i\right)},s_{i}\right)}\left[\kappa^{\left(i\right)}\right]\ot_{\O_{\X}}\gr^{h}F^{\b}\O_{\V_{\O_{F_{i}}}\left(\left(\ud^{p}\cdot\uo_{\A^{D}}^{\lor}\right)^{\left(i\right)}\right)}\\
 & =\frak{w}^{\kappa^{\left(i\right)}}\ot_{\O_{\X}}\uo_{\A}^{-h}\ot_{\O_{\X}}\uo_{\A^{D}}^{-h}.
\end{align*}
\end{proof}

\subsection{A connection on $\mathbb{W}_{\kappa}$\label{subsec:The-Gauss-Manin-connection-1}}
\begin{lem}
\label{lem:connessionescendesuIG}The Gauss-Manin connection $\nabla_{{\rm GM}}:H_{\dR}^{1}\left(\A/\X\right)\to H_{\dR}^{1}\left(\A/\X\right)\ot_{\O_{\X}}\Omega_{\X/B_{\ga,I}^{0}}^{1}$
induces an integrable ${\rm MS}_{\O_{F_{i}}}$-connection 
\[
\nabla^{\#,\left(i\right)}:H_{\A}^{\#,\left(i\right)}\to H_{\A}^{\#,\left(i\right)}\ot_{\O_{\IG{}_{n}^{\prime\left(i\right)}}}\Omega_{\IG{}_{n}^{\prime\left(i\right)}/B_{\ga,I}^{0}}^{1}.
\]
\end{lem}

\begin{proof}
It follows from Proposition \ref{prop:gaussmaninrispettaO_F} and
\cite[Proposition 6.3, pag. 66]{AI}.
\end{proof}
In view of Lemma \ref{lem:connessionescendesuIG} we have an integrable
connection
\[
\nabla_{\s,u}^{\left(i\right)}:\pi_{\ast}\O_{\V_{0,\s}^{u}\left(H_{\A}^{\#,\left(i\right)},s_{i}\right)}\to\pi_{\ast}\O_{\V_{0,\s}^{u}\left(H_{\A}^{\#,\left(i\right)},s_{i}\right)}\widehat{\ot}_{\O_{\IG_{n}^{\prime}{}^{\left(i\right)}}}\Omega_{\IG_{n}^{\prime}{}^{\left(i\right)}/B_{\ga,i}^{0}}^{1}.
\]
We want to see how the connection $\nabla_{\s,u}^{\left(i\right)}$
descends to $\IG_{n}^{\left(i\right)}$. Recall that locally we have
\[
H_{\A}^{\#}=\Omega_{\A}+\delta^{p}H_{\dR}^{1}\left(\A/\IG_{n}\right)\su H_{\dR}^{1}\left(\A/\IG_{n}\right)
\]
so let us check that 
\[
\nabla\left(\delta^{p}H_{\dR}^{1}\left(\A/\IG_{n}\right)\right)\su\delta^{p}H_{\dR}^{1}\left(\A/\IG_{n}\right)\ot_{\O_{\IG_{n}}}\Omega_{\IG_{n}/B_{\ga,I}^{0}}^{1}.
\]
For $x\in H_{\dR}^{1}\left(\A/\IG_{n}\right)$ we have
\[
\nabla\left(\delta^{p}x\right)=\delta^{p}\nabla\left(x\right)+p\delta^{p-1}x\ot d\delta,
\]
but by its very construction we have $p\in\left(\delta\right)$, whence
$p\delta^{p-1}\su\left(\delta\right)^{p}$. A problem arises when
we take $\Omega_{\A}$ into account: recall that $\gd\cdot\uo_{\A}=\Omega_{\A}$.
Let $\spf\left(R\right)\su\IG_{n}$ be an open affine subset with
preimage $\spf\left(R^{\prime}\right)\su\IG_{n}'$ and such that its
image in $\X$ is contained in an open affine $\spf\left(A\right)$
over which the sequence $\HH_{\A}^{\b}$ is split. Let $\omega,\eta,\gd$
be bases of $\uo_{\A|R},\uo_{\A|R}^{-1}$ and $\underline{\gd}_{R}$
respectively, then $\left\{ \omega,\eta\right\} $ is a basis for
$H_{\dR}^{1}\left(\A/\IG_{n}\right)_{|R}$ and $\left\{ \gd\omega,\delta^{p}\eta\right\} $
is a basis for $H_{\A|R}^{\#}$. We have a Kodaira-Spencer isomorphism
\[
{\rm KS}:\uo_{\A}\to\uo_{\A}^{-1}\ot_{\O_{\X}}\Omega_{\X/B_{\ga,I}^{0}}^{1}
\]
where we recall that the $\O_{L}$-structure is the one induced by
this isomorphism, namely
\[
\Omega_{\X/B_{\ga,I}^{0}}^{1}\simeq\uo_{\A}^{\ot2}
\]
 and a generator $\theta_{\s}^{\left(i\right)}$ of $\Omega_{\X/B_{\ga,I}^{0}}^{1,\left(i\right)}\left(\s\right)$
over $A$ characterised by the property that 
\[
{\rm KS}\left(\omega_{\s}^{\left(i\right)}\right)=\eta_{\s}^{\left(i\right)}\ot\theta_{\s}^{\left(i\right)}.
\]
We have elements $m_{\s}^{\left(i\right)},t_{\s}^{\left(i\right)},s_{\s}^{\left(i\right)}\in A$
such that
\[
\nabla^{\left(i\right)}:\begin{cases}
\omega_{\s}^{\left(i\right)}\mapsto m{}_{\s}^{\left(i\right)}\omega{}_{\s}^{\left(i\right)}\ot\theta{}_{\s}^{\left(i\right)}+\eta{}_{\s}^{\left(i\right)}\ot\theta{}_{\s}^{\left(i\right)}\\
\eta{}_{\s}^{\left(i\right)}\mapsto t{}_{\s}^{\left(i\right)}\omega{}_{\s}^{\left(i\right)}\ot\theta{}_{\s}^{\left(i\right)}+s{}_{\s}^{\left(i\right)}\eta{}_{\s}^{\left(i\right)}\ot\theta{}_{\s}^{\left(i\right)}
\end{cases}
\]
hence 
\begin{align}
\nabla^{\#}\left(\delta^{\left(i\right)}\omega{}_{\s}^{\left(i\right)}\right) & =\delta^{\left(i\right)}\nabla^{\#}\left(\omega{}_{\s}^{\left(i\right)}\right)+\omega{}_{\s}^{\left(i\right)}\ot d\delta^{\left(i\right)}\label{eq:nabla=000023}\\
 & =m_{\s}^{\left(i\right)}\delta^{\left(i\right)}\omega{}_{\s}^{\left(i\right)}\ot\theta{}_{\s}^{\left(i\right)}+\delta^{\left(i\right)}\eta{}_{\s}^{\left(i\right)}\ot\theta{}_{\s}^{\left(i\right)}+\omega{}_{\s}^{\left(i\right)}\ot d\delta^{\left(i\right)}\nonumber \\
\nabla^{\#}\left(\left(\delta^{\left(i\right)}\right)^{p}\eta{}_{\s}^{\left(i\right)}\right) & =\left(\delta^{\left(i\right)}\right)^{p}\nabla^{\#}\left(\eta{}_{\s}^{\left(i\right)}\right)+\eta{}_{\s}^{\left(i\right)}\ot p\left(\delta^{\left(i\right)}\right)^{p-1}d\delta^{\left(i\right)}\nonumber \\
 & =t_{\s}^{\left(i\right)}\left(\delta^{\left(i\right)}\right)^{p}\omega{}_{\s}^{\left(i\right)}\ot\theta{}_{\s}^{\left(i\right)}+s{}_{\s}^{\left(i\right)}\left(\delta^{\left(i\right)}\right)^{p}\eta{}_{\s}^{\left(i\right)}\ot\theta{}_{\s}^{\left(i\right)}+\left(\delta^{\left(i\right)}\right)^{p}\eta{}_{\s}^{\left(i\right)}\ot p\dlog\delta^{\left(i\right)},\nonumber 
\end{align}
In view of the term $\dlog\delta^{\left(i\right)}$ we conclude that
the connection $\nabla^{\#,\left(i\right)}$ descends to an integrable
${\rm MS}_{\O_{F_{i}}}$-connection
\begin{equation}
\nabla^{\#,\left(i\right)}:H_{\A}^{\#,\left(i\right)}\to H_{\A}^{\#,\left(i\right)}\ot_{\O_{\IG_{n}^{\left(i\right)}}}\Hdg^{-1}\cdot\Omega_{\IG{}_{n}^{\left(i\right)}/B_{\ga,I}^{0}}^{1}.\label{eq:=00005Cnabla=000023(i)}
\end{equation}
\begin{prop}
\label{prop:esisteconnessionesuW}The integrable ${\rm MS}_{\O_{F_{i}}}$-connection
\[
\nabla_{\s}^{\#,\left(i\right)}:H_{\A}^{\#,\left(i\right)}\left(\s\right)\to H_{\A}^{\#,\left(i\right)}\left(\s\right)\ot_{\O_{\IG_{n}^{\left(i\right)}}}\Hdg^{-1}\cdot\Omega_{\IG{}_{n}^{\left(i\right)}/B_{\ga,I}^{0}}^{1}
\]
induces integrable connections
\[
\nabla_{\kappa^{\left(i\right)}}:\tilde{\mathbb{W}}_{\kappa^{\left(i\right)},\s,u}\to\tilde{\mathbb{W}}_{\kappa^{\left(i\right)},\s,u}\widehat{\ot}_{\O_{\IG_{n}^{\left(i\right)}}}\Hdg^{-1}\cdot\Omega_{\IG_{n}^{\left(i\right)}/B_{\ga,I}^{0}}^{1}
\]
and
\[
\nabla_{\kappa^{\left(i\right)}}:\mathbb{W}_{\kappa^{\left(i\right)},\s,u}\to\mathbb{W}_{\kappa^{\left(i\right)},\s,u}\widehat{\ot}_{\O_{\X}}\Hdg^{-1}\cdot\Omega_{\X/B_{\ga,I}^{0}}^{1}
\]
that respects Griffith transversality property for the filtration
$F^{\b}\tilde{\mathbb{W}}_{\kappa^{\left(i\right)},\s,u}$. Moreover
the induced map on graded pieces
\[
\gr^{h}F^{\b}\nabla_{\kappa^{\left(i\right)}}:\ga^{-1}{\rm Gr}^{h}F^{\b}\mathbb{W}_{\kappa^{\left(i\right)},\s,u}\to\ga^{-1}{\rm Gr}^{h+1}F^{\b}\mathbb{W}_{\kappa^{\left(i\right)},\s,u}\widehat{\ot}_{\O_{\X_{r,I}^{L}}}\Omega_{\X^{L}/\left(B_{\ga,I}^{0}\ot\O_{L}\right)}^{1}
\]
is the composition of an isomorphism and the product by $u_{I}^{\left(i\right)}-h$
(cfr. Note \ref{note:esisteelementou_I}).
\end{prop}

\begin{proof}
Note that the map $\IG_{n}^{\left(i\right)}\to\X^{\left(i\right)}$
is finite étale after inverting $\Hdg$ with Galois group $\left(\O_{F_{i}}/p^{n}\right)\inv$,
therefore faithfully flat descent applies to $\nabla_{\kappa^{\left(i\right)}}$
and we reduce to the proof of the existence of the connection over
$\IG_{n}^{\left(i\right)}$. Let $\spf\left(R\right)\su\IG_{n}^{\left(i\right)}$
be an open affine over which $H_{\A}^{\#,\left(i\right)}$ is free,
say with basis $f_{\s}^{\left(i\right)}=\delta^{\left(i\right)}\omega{}_{\s}^{\left(i\right)},e{}_{\s}^{\left(i\right)}=\left(\delta^{\left(i\right)}\right)^{p}\eta{}_{\s}^{\left(i\right)}$,
and let $\spf\left(R'\right)\su\IG_{n}^{\prime}{}^{\left(i\right)}$
be its inverse image. Using notations as in \ref{subsec:Connections-on-the}
we have an ${\cal A}_{R'/A}$-linear isomorphism
\[
\epsilon{}_{\s}^{\left(i\right)}:H_{\A}^{\#,\left(i\right)}\left(\s\right)\ot_{j_{2}}{\cal A}_{R'/A}\to H_{\A}^{\#,\left(i\right)}\left(\s\right)\ot_{j_{1}}{\cal A}_{R'/A}
\]
which corresponds to a matrix
\[
\left(\begin{matrix}a_{\s}^{\left(i\right)} & b_{\s}^{\left(i\right)}\\
c_{\s}^{\left(i\right)} & d_{\s}^{\left(i\right)}
\end{matrix}\right)\in{\rm GL}_{2}\left({\cal A}_{R'/A}\right)
\]
in terms of the basis we picked. The condition that $\epsilon{}_{\s}^{\left(i\right)}\ot_{\Delta}R'={\rm id}_{H_{\A}^{\#,\left(i\right)}\left(\s\right)}$
translates into an equality
\[
\left(\begin{matrix}a_{\s}^{\left(i\right)} & b_{\s}^{\left(i\right)}\\
c_{\s}^{\left(i\right)} & d_{\s}^{\left(i\right)}
\end{matrix}\right)=\left(\begin{matrix}\left(1,\omega_{a_{\s}^{\left(i\right)}}\right) & \left(0,\omega_{b_{\s}^{\left(i\right)}}\right)\\
\left(0,\omega_{c_{\s}^{\left(i\right)}}\right) & \left(1,\omega_{d_{\s}^{\left(i\right)}}\right)
\end{matrix}\right).
\]
This gives
\[
\nabla^{\#}\left(f_{\s}^{\left(i\right)}\right)=f{}_{\s}^{\left(i\right)}\ot\left(0,\omega_{a_{\s}^{\left(i\right)}}\right)+e{}_{\s}^{\left(i\right)}\ot\left(0,\omega_{c_{\s}^{\left(i\right)}}\right),
\]
on the other hand we have formula \ref{eq:nabla=000023} telling that
\[
\nabla^{\#}\left(f_{\s}^{\left(i\right)}\right)=m{}_{\s}^{\left(i\right)}f{}_{\s}^{\left(i\right)}\ot\theta{}_{\s}^{\left(i\right)}+\gd^{\left(i\right)}\eta{}_{\s}^{\left(i\right)}\ot\theta{}_{\s}^{\left(i\right)}+\omega{}_{\s}^{\left(i\right)}\ot d\gd^{\left(i\right)}.
\]
Since $\omega{}_{\s}^{\left(i\right)}$and $\eta{}_{\s}^{\left(i\right)}$
are linearly independent we conclude that 
\[
\theta{}_{\s}^{\left(i\right)}=\left(\delta^{\left(i\right)}\right)^{p-1}\omega_{c_{\s}^{\left(i\right)}}
\]
and note that this equality is over $A$ (that is, over $\X$) since
$\left(\delta^{\left(i\right)}\right)^{p-1}$ is in $\Hdg$ (it is
a generator indeed). This shows that $\left(\delta^{\left(i\right)}\right)^{p-1}\omega_{c_{\s}^{\left(i\right)}}$
is a local generator of $\Omega_{\IG_{n}^{\left(i\right)}/B_{\ga,I}^{0}}^{1}\left(\s\right)$.
Consider the local sections $X_{1,\s},X_{2,\s}$ of $\V_{\O_{F_{i}}}\left(H_{\A}^{\#,\left(i\right)}\left(\s\right)\right)$
obtained from $f_{\s}^{\left(i\right)}$ and $e_{\s}^{\left(i\right)}$.
The action of $\epsilon{}_{\s}^{\left(i\right)}$ on them is given
by
\[
\epsilon{}_{\s}^{\left(i\right)}\left(\begin{matrix}X\\
Y
\end{matrix}\right)=\left(\begin{matrix}a_{\s}^{\left(i\right)}X+b{}_{\s}^{\left(i\right)}Y\\
c_{\s}^{\left(i\right)}X+d{}_{\s}^{\left(i\right)}Y
\end{matrix}\right),
\]
then, setting $V=X_{2,\s}X_{1,\s}^{-1}$ and keeping in mind that
the ideal $\Omega_{\X/B_{\ga,I}^{0}}^{1}\su{\cal A}$ is a square-zero
ideal, we compute
\begin{align*}
\epsilon{}_{\s}^{\left(i\right)}\left(\kappa^{\left(i\right)}\left(X_{1,\s}\right)V_{\s}^{h}\right) & =\kappa^{\left(i\right)}\left(\e{}_{\s}^{\left(i\right)}\left(X_{1,\s}\right)\right)\left(\frac{c_{\s}^{\left(i\right)}X_{1,\s}+d{}_{\s}^{\left(i\right)}X_{2,\s}}{a_{\s}^{\left(i\right)}X_{1,\s}+b{}_{\s}^{\left(i\right)}X_{2,\s}}\right)^{h}\\
 & =\kappa^{\left(i\right)}\left(X_{1,\s}\right)\kappa^{\left(i\right)}\left(a_{\s}^{\left(i\right)}+b{}_{\s}^{\left(i\right)}V_{\s}\right)\left(c_{\s}^{\left(i\right)}+d{}_{\s}^{\left(i\right)}V_{\s}\right)^{h}\left(a_{\s}^{\left(i\right)}+b{}_{\s}^{\left(i\right)}V_{\s}\right)^{-h}\\
 & =\kappa^{\left(i\right)}\left(X_{1,\s}\right)\exp\left(u_{I}^{\left(i\right)}\log\left(\left(1,0\right)+\left(\left(0,\omega_{a_{\s}^{\left(i\right)}}\right)+\left(0,\omega_{b_{\s}^{\left(i\right)}}\right)V_{\s}\right)\right)\right)\\
 & \quad\left(\left(0,\omega_{c_{\s}^{\left(i\right)}}\right)+\left(1,\omega_{d_{\s}^{\left(i\right)}}\right)V_{\s}\right)^{h}\left(\left(1,\omega_{a_{\s}^{\left(i\right)}}\right)+\left(0,\omega_{b_{\s}^{\left(i\right)}}\right)V_{\s}\right)^{-h}\\
 & =\kappa^{\left(i\right)}\left(X_{1,\s}\right)\left(\left(1,0\right)+u_{I}\left(\left(0,\omega_{a_{\s}^{\left(i\right)}}\right)+\left(0,\omega_{b_{\s}^{\left(i\right)}}\right)V_{\s}\right)\right)\\
 & \quad\left(\left(1,h\omega_{d}\right)V_{\s}^{h}+\left(0,h\omega_{c}\right)V_{\s}^{h-1}\right)\left(\left(1,0\right)-h\left(\left(0,\omega_{a}\right)+\left(0,\omega_{b_{\s}^{\left(i\right)}}\right)V_{\s}\right)\right)\\
 & =\kappa^{\left(i\right)}\left(X_{1,\s}\right)\left(V^{h}+\left(0,h\omega_{d_{\s}^{\left(i\right)}}\right)V_{\s}^{h}-\left(0,h\omega_{a_{\s}^{\left(i\right)}}\right)V_{\s}^{h}-\left(0,h\omega_{c}\right)V_{\s}^{h+1}\right.\\
 & \quad\left.+\left(0,h\omega_{c}\right)V_{\s}^{h-1}+\left(0,u\omega_{a}\right)V_{\s}^{h}+\left(0,u\omega_{c}\right)V_{\s}^{h+1}\right),
\end{align*}
hence
\begin{align}
\nabla^{\#}\left(\kappa^{\left(i\right)}\left(X_{1,\s}\right)V_{\s}^{h}\right) & =\e\left(\kappa^{\left(i\right)}\left(X_{1,\s}\right)V_{\s}^{h}\right)-\kappa^{\left(i\right)}\left(X_{1,\s}\right)V_{\s}^{h}\equiv\label{eq:descrzionenabla}\\
 & \equiv\left(u_{I}^{\left(i\right)}-h\right)V_{\s}^{h+1}\ot\omega_{c_{\s}^{\left(i\right)}}\in{\rm Gr}^{h+1}F^{\b}\mathbb{W}_{\kappa^{\left(i\right)},\s,u}\widehat{\ot}_{\O_{\X}}\Omega_{\X/B_{\ga,I}^{0}}^{1}.
\end{align}
The form $\omega_{c_{\s}^{\left(i\right)}}$ generates $\ga^{-1}\Omega_{\X/B_{\ga,I}^{0}}^{1}$,
hence, after pulling back to $\V_{0,\s}^{u}\left(H_{\A}^{\#,\left(i\right)},s_{i}\right)$
(corresponding to $X_{1,\s}\mapsto\s\left(u\right)+\gb_{n}Z_{\s,u}$)
we deduce that $\gr^{h}F^{\b}\nabla_{\kappa^{\left(i\right)}}$ gives
an isomorphism
\[
\ga^{-1}{\rm Gr}^{h}F^{\b}\mathbb{W}_{\kappa^{\left(i\right)},\s,u}\simeq\ga^{-1}\left(u_{I}^{\left(i\right)}-h\right){\rm Gr}^{h+1}F^{\b}\mathbb{W}_{\kappa^{\left(i\right)},\s,u}\widehat{\ot}_{\O_{\X}}\Omega_{\X/B_{\ga,I}^{0}}^{1}.
\]
\end{proof}
\newpage{}

\end{document}